\documentclass[12pt]{amsart} 

\usepackage{amsmath} \usepackage{amssymb} \usepackage{amsthm}
\usepackage{amscd} 
\def\HH{{\mathrm{H}}}

\def\OO{{\mathcal{O}}}\def\PP{{\mathbb{P}}}
\def\Q{{\mathbb{Q}}}  \def\C{{\mathbb{C}}}

\def\Spec{{\mathrm{Spec\; }}} \def\Proj{{\mathrm{Proj\; }}}

\def\dim{{\mathrm{dim}}}\def\deg{{\mathrm{deg}}}
\def\divv{{\mathrm{div}}}
\def\FRC{{\mathrm{frac}}}
\def\mod{{\mathrm{mod}}}
\def\LLR{\Longleftrightarrow}
\theoremstyle{plain}
\newtheorem{thm}{Theorem}[section] 

\newtheorem{prop}[thm]{Proposition}

\newtheorem{lem}[thm]{Lemma}
\theoremstyle{definition} 
\newtheorem{defn}[thm]{Definition}
 
\theoremstyle{remark}
\newtheorem{rem}[thm]{Remark}
\newtheorem{ques}[thm]{Question}

\title{Classification of 2-dimensional graded normal hypersurfaces 
with $a(R)\le 6$.}
\author[Watanabe]{Kei-ichi Watanabe} 
\address{Department of Mathematis, College of Humanities and Sciences, Nihon University Setagaya, Tokyo 156-8550, Japan}
\email{watanabe@math.chs.nihon-u.ac.jp}
\thanks{This work was partially supported by JSPS Grant-in-Aid for Scientific Research  20540050 and Individual Research Expense of College of Humanity and Sciences, Nihon University.}
\begin{document}

\maketitle

\section*{Introduction}

Isolated weighted homogeneous surface singularities 
of type $R=k[X,Y,Z]/(f)$ are extensively studied by 
V.I. Arnold, H. Pinkham \cite{P1} and K. Saito and many other authors. 
Especially, K. Saito \cite{S1} studies these in terms of \lq\lq regular 
system of weights".  On the other hand, from the view point of 
commutative ring theory, the $a$-invariant $a(R)$ defined in \cite{GW}
 (in Saito's paper, the notation $\varepsilon = -a(R)$ is used) and 
 also, such singularities can be constructed by so called DPD 
 (Dolgacev-Pinkham-Demazure) constrction \cite{Dol}\cite{P2}\cite{Dem}. 
\par   
The author tried the classification of all the possible weights 
for given $a(R)$ using DPD construction of normal graded rings 
and  it turned out that the procedure is so simple and nearly automatic. 
\par
Although this classification is \lq\lq known" in the literature (cf. \cite{S1}, 
\cite{Wag}), it seems that the algebraic or ring-theoretic classification is not done yet.
\par
Another good point of our classification is that we can draw the \lq\lq graph" of the 
resolution of the singularity of $\Spec(R)$ instantly from the expression as 
$R = R(X,D)$. 
\par
Below, we present the classification of such singularities with $0<a(R)\le 6$. 
We prove in general that for a given $\alpha>0$, the number of types of $R$ for 
2-dimensional normal  graded hypersurfaces with $a(R)=\alpha$ is finite.

\section{Preliminaries}

Let $R=\oplus_{n\ge 0}R_n= k[u,v,w]/(f)$ be a $2$-dimensional graded normal 
hypersurface, 
where $k$ is an algebraically closed field of any characteristic.  
We always assume that the grading of $R$ is given so that $R_n\ne 0$
 for $n\gg 0$. 
We put $X=\Proj(R)$.  Since $\dim R=2$ and $R$ is normal, $X$ is a 
smooth curve.  Then by the construction of Dolgachev, Pinkham  and Demazure (\cite{Dem},
\cite{P2}),  there is 
an ample $\Q$-Cartier divisor $D$ (that is, $ND$ is an ample divisor on 
$X$ for some positive integer $N$), such that 
$$R\cong R(X,D) = \oplus_{n\ge 0}H^0(X,O_X(nD))\cdot T^n\subset k(X)[T]$$
as graded rings, where $T$ is a variable over $k(X)$ and 
$$H^0(X,O_X(nD))=\{f\in k(X)\;|\;\divv_X(f)+nD\ge 0\}\cup\{0\}.$$

We denote by $[nD]$ the integral part of $nD$ so that $H^0(X,O_X(nD))= 
H^0(X,O_X([nD]))$. 

Now, let us begin the classification. In the following, 
 $X$ is a smooth curve of genus $g$ and $D$ is a fractional divisor
  on $X$ such that $ND$ is an ample integral (Cartier) divisor for some $N>0$. 
\par
We  always denote
\[\leqno{(1.0.1)}\quad D=E+ \sum_{i=1}^r \dfrac{p_i}{q_i} P_i \quad (p_i, q_i \;{\mathrm{are \; positive  \; integers \; and \;}}\forall i,\; (p_i,q_i)=1),\]
where $E$ is an integral divisor. In this case, we denote 
\[\leqno{(1.0.2)}\quad
\FRC(D)=\sum_{i=1}^r \dfrac{q_i-1}{q_i} P_i. \]

At the same time, by our assumption $R\cong k[u,v,w]/(f)$. If 
$\deg(u,v,w;f)=(a,b,c;h)$, then by \cite{GW}, 
\[\leqno{(1.0.3)}\quad a(R)= h-(a+b+c). \]

We always assume $\deg(u,v,w;f)=(a,b,c;h)$ and also that 
 $a\le b\le c$. We call $(a,b,c; h)$ the {\it type} of $R$.  In this paper, 
we will determine all the possible types of $R$ for $0< a(R)\le 6$
\footnote{See Remark \ref{negative} for the case $\alpha <0$}. \par

We list our tools to classify. 

\begin{prop}\label{Fund}({\bf Fundamental formulas}) 
Assume that $R=R(X,D)$\par
\noindent $
\cong k[u,v,w]/(f)$ with $\deg(u,v,w;f)=(a,b,c;h)$ 
and $a(R)=h-(a+b+c)=\alpha$. Then we have the following equalities. \par
(1) \cite{W} Since $R$ is Gorenstein  with $a(R)=\alpha$, we have 
$$\alpha D\sim K_X + \FRC(D) = K_X +\sum_{i=1}^r \dfrac{q_i-1}{q_i} P_i,$$
where, in general, $D_1\sim D_2$ means that $D_1-D_2=\divv_X(\phi )$ for 
some $\phi \in k(X)$.\par
In particular, the genus $g$ of $X$ is given by 
$$g=g(X) = H^0(X, O_X(K_X))= \dim R_{\alpha}.$$

(2) \cite{To}  If $P(R,t)=\sum_{n\ge 0} \dim R_n t^n$ is the Poincare series of R, then 
$$\lim_{t\to 1} (1-t)^2P(R,t) = \deg D.$$
 Since $P(R,t)=\dfrac{1-t^h}{(1-t^a)(1-t^b)(1-t^c)}$  in our case, we have 
$$\deg D= \dfrac{h}{abc}=\dfrac{\alpha}{abc}+\dfrac{1}{ab}+\dfrac{1}{ac}+
\dfrac{1}{bc}.$$

Note that the latter expression is a decreasing function of $a,b,c$. 
\end{prop}

\begin{lem}\label{bd-of-c}\cite{OW}  Let $R=k[x,y,z]/(f)$ be a normal graded ring with 
type $(a,b,c;h)$ with $h=a+b+c+\alpha $.  Then  \par
(1) If $h$ is not a multiple of $c$, then either $h-a$ or $h-b$ is a multiplie of $c$. 
The same is true  for $a,b$. Namely, at least one of $h, h-a,h-c$ is a 
multiple of 
$b$ and at least one of $h,h-b,h-c$ is a multiple of $a$. \par
(2) If a prime number $p$ divides two of $a,b,c$, then $p$ divides $h$. 
In particular, if $\alpha$ is even, then at most one among $a,b,c$ is even. \par
(3) $c\le a+b+\alpha $. 
\end{lem}

\begin{proof} If $R=k[x,y,z]/(f)$ is normal, then $f$ must contain monomial of the 
form $x^n$ or $x^ny$ or $x^nz$. The statement (1), (3) follows from this fact.  As 
for (2), It is easy to see that if $p$ divides $a,b$ and not $h$, then $f$ must be 
divisible by $z$. 
\end{proof}

We list some properties of $R$ when $R$ is Gorenstein.

\begin{lem}\label{dual} Let $R=R(X,D)$ be a normal graded ring with 
$D$ as in (1.0.1) and 
assume that $R$ is Gorenstein with $a(R)=\alpha$.  Then we have the following formulas.
\begin{enumerate}
\item 
For every 
$n$, $0\le n\le \alpha$, we have $\deg [nD]+\deg [(\alpha -n)D]
=\deg [\alpha D] = 2g -2$, where  $g$ is the genus of $X$.\par
\item $[(\alpha + 1)D] = K_X + E + \sum_{i=1}^r P_i$. 
\item $[(2\alpha + 1)D] = 2K_X + E + 2\sum_{p_i \ge 2} P_i + \sum_{p_i=1} P_i$.  
\end{enumerate}
\end{lem}
\begin{proof} This follows easily from $\alpha D=K_X+\FRC(D)$. 
\end{proof}

Next we recall some fundamental property of $p_g(R)$. \par

\begin{defn}
If $X \to \Spec(R)$ is a resolution of singularities of $R$, then the geometric 
genus of $R$, $p_g(R)$ is defined by 
$$p_g(R)=\dim_k \HH^1(X, O_X).$$
\end{defn}

When $R$  is a Gorenstein graded ring, it is proved in \cite{W} that

\[\leqno{(1.5.1)}\quad
\qquad p_g(R)=\sum_{n=0}^{a(R)} \dim_k R_n.\]

In the following, we denote $a(R)=\alpha$ to avoid confusion with $a$ 
(the minimal positive degree with $R_a\ne 0$).

\begin{rem}\label{resol} \cite{Dem} If $R=R(X,D)$ with $D=E+ \sum_{i=1}^r \dfrac{p_i}{q_i} P_i$, 
the \lq\lq graph" of the resolution of singularity of $\Spec(R)$ is a so called \lq\lq 
star-shaped" graph with \lq\lq central curve" $X$ with self intersection $X^2 = - \lceil D\rceil$ and 
branch of ${\Bbb P}^1$'s intersecting at $P_i\in X$  with self intersection number 
$-b_1,\ldots , -b_s$ if the continued-fraction expression of $\dfrac{q_i}{q_i-p_i}$ is as  
$\dfrac{q_i}{q_i-p_i}= b_1- \dfrac{1}{b_2- \dfrac{1}{b_3- \ldots }}$
\end{rem}

\section{General results for  given   $\alpha = a(R)>0$.}

First, we will show the finiteness of the types $(a,b,c; h)$ for given $\alpha =a(R)>0$. 

\begin{thm}\label{finite}
If we fix $\alpha =a(R) > 0$, the number of 
$(a,b,c; h)$ for a normal graded ring $R=k[x,y,z]/(f)$ with $h=
a+b+c+\alpha $ is finite
\footnote{See Remark \ref{negative} for the case $\alpha \le 0$.}.
\end{thm}
  
In the rest of this section, we fix $\alpha >0$, always assume that $a\le b\le c, h=a+b+c+\alpha $
 and $(a,b,c)=1$.  The proof is done in a series of lemmas.
Note that by Lemma \ref{bd-of-c}, it suffices to show that the number of possible $(a,b)$ 
is finite. Also, for  simplicity, we assume that $c\ge \alpha$. 

\begin{lem}\label{degDle4ab} (1) We have $\deg D \le \dfrac{4}{ab}$.\par
(2) If $g\ge 2$, then $ab< \dfrac{4\alpha }{2g-2}$.
\end{lem} 
\begin{proof} (1) By our assumption $c\ge \alpha$, 
 $\deg D=\dfrac{a+b+c+\alpha }{abc}\le \dfrac{4}{ab}$. \par
 (2) Since $\alpha D\sim K_X+\FRC(D)$, $\deg D  \ge \dfrac{2g-2}{\alpha }$ and 
the LHS is less than $\dfrac{4}{ab}$.
 \end{proof}

\begin{lem} If $g=1$, then $ab< 8\alpha $.
\end{lem} 
\begin{proof} As in the previous lemma, since $\alpha\; \deg D \sim \FRC(D)$ and 
$\deg \;\FRC(D)\ge \dfrac{1}{2}$, we have 
$\deg D=\dfrac{a+b+c+\alpha }{abc} \ge \dfrac{1}{2\alpha }$. 
Hence we have  
$\dfrac{4}{ab}> \dfrac{1}{2\alpha}$. 
\end{proof}

\begin{lem} If $g=0$, then $ab< 168\alpha $.
\end{lem} 
\begin{proof} We have $\alpha \deg D = \FRC(D) -2$ and it is easy to show that 
if $\sum_{i=1}^r \dfrac{q_i-1}{q_i} -2>0$, then the minimal such value 
of the LHS is $\dfrac{1}{42}$. 
\end{proof}

By these Lemmas, we have proved  that the number of types of $R$ is finite for a fixed $\alpha=a(R)$. 

\begin{rem}\label{negative}  Although the following results are somewhat \lq\lq known"
 in the literature,
I include the cases of $\alpha= a(R)\le 0$ for the completeness. \par
If $\alpha =0$, then $K_X=0$ and $\FRC(D)=0$. Hence $g(X)=1$ and $\deg D\le 3$.
We have $(a,b,c;h)= (1,2,3;6), (1,2,2;4),(1,1,1;3)$ according to $\deg D=1,2,3$ respectively.\par

If $\alpha <0$, then since $\deg(K_X + \FRC(D)) <0$,   we have $X\cong \PP^1,  
r\le 3$ and if $r=3$, 
$(q_1,q_2,q_3)$ is either $(2,2,n),(2,3,3), (2,3,4)$ or $(2,3,5)$.  If we put $D= 
-(K_X + \FRC(D))$ in these cases, then we have $\alpha =-1$ and we get  
$(a,b,c;h)=$\par
\noindent $(2,n,n+1; 2n+2), (3,4,6;12), 
(4,6,9;18), (6,10,15)$ respectively, corresponding to $(D_{n+2}), (E_6), (E_7), (E_8)$ 
singularities.  Note that the number of types is infinite in this case. \par
Since in the case $\alpha<0$ and $r=3$, there is no integer $n<-1$ with 
$nD \sim K_X + \FRC(D)$,  if $\alpha <-1$, then $r\le 2$ and putting $P_1=(0)$ and 
$P_2=(\infty)$, then $R$ is generated by monomials and is isomorphic to a ring of the form 
$k[u,v,w]/(uv - w^n)$.  Note that in this case, the type $(a,b,c;h)$ is not uniquely determined 
by the ring. On the contrary, if $\alpha \ge 0$ or $r\ge 3$, the type   $(a,b,c;h)$ is uniquely determined by the ring $k[X,Y,Z]/(f)$ since the resolution of $\Spec(R)$ given in \ref{resol}
is a minimal good resolution\footnote{A resolution whose exceptional set consists of smooth 
curves with normal crossings and  which is minimal resolution satisfying this condition; 
namely, which contains no $(-1)$ curve intersecting to at most $2$ 
other irreducible curves.} and conversely a minimal good resolution determines $R$ as a 
graded ring.  Related general statement can be found in \cite{S0}.
\end{rem}

In the following, we consider only the case $\alpha >0$.  

\begin{ques} Assume that $R=k[X,Y,Z,W]/(f)$ is graded with type $(a,b,c,d; h)$ 
and has isolated singularity. For a given $\alpha >0$, is the number of $(a,b,c,d; h)$ 
with $\alpha = h - (a+b+c+d)$ finite ? 
\end{ques}

The following Theorem is the main result of \cite{S2}.  We give a proof  here because 
it is much simpler by our method than the one given there.

\begin{thm}\label{alpha+1} For any given $\alpha >0$, we have either $R_{\alpha -1}\ne 0$ or 
$R_{\alpha +1}\ne 0$. 
\end{thm} 

\begin{proof} Let $R=R(X,D)$ and $g$ be the genus of $X$ and put 
$D= E + \sum_{i=1}^r \dfrac{p_i}{q_i} P_i$ as in (1.0.1).  By Lemma \ref{dual} (2), 
we have always
\[ [(\alpha + 1) D] = K_X + E + \sum_{i=1}^r P_i.\]
Since $\deg D>0$, we have $\deg E \ge 1-r $ and $\deg [ (\alpha + 1) D] 
\ge 2g-1$.  Hence if $g>0$, we have always $R_{\alpha +1} \ne 0$.\par
If $g=0$, we have 
\[\deg([ (\alpha - 1) D] ) = -2 - \deg E \quad {\rm and} \quad 
\deg([ (\alpha + 1) D] ) = - 2 + \deg E + r.\]
Hence if $\deg E\le -2$, we have $R_{\alpha -1} \ne 0$ and  if $\deg E \ge -1$, 
then $R_{\alpha +1}\ne 0$ since $r\ge 3$.
\end{proof}

Next we will show that if $p_g(R)=1$, then $\alpha\le 7$ and  that $\alpha$ is bounded if 
$p_g(R)$ is bounded. These result also appears in Saito's paper \cite{S1}, 
but we will show it by our method.
 We can easily distinguish $R$ with $p_g(R)=1$ and $\alpha\le 6$ from our table.

\begin{thm}\label{pg} (1)  If $p_g(R)=r$ is fixed, then $\alpha$ is bounded.\par 
(2) If $p_g(R)=1$, then $\alpha \le 7$. If $p_g(R)=1$ and $\alpha =7$, then 
the type of $R$ is either $(8,10,15; 40), (8,10,25; 50)$ or $(8,9,12; 36)$. 
\end{thm}

\begin{proof} (1)  If $p_g(R)=r$, then $a\le \alpha/r$ by formula (1.5.1). 
Hence if $\alpha$ tends to $\infty$, then   $\deg D \le c \alpha^{-2}$ for some constant 
$c$ by Proposition 1.1. While we have shown in Lemma 2.2 to 2.4 that $\deg D \ge c' \alpha^{-1}$ 
for some constant $c'$. \par
(2) Put $R=R(X,D)$ and $g$ be the genus of $X$.  If $g>0$, 
then we have $\dim R_{\alpha} =g$ and hence we have always 
$p_g(R) \ge g+1 \ge 2$.\par
Now we assume $g=0$ and $p_g(R)=1$ and as always, we write 
$D = E + \sum_{i=1}^r \dfrac{p_i}{q_i} P_i$. 
By formula (1.5.1), we have $R_n=0$ for $1\le n \le \alpha$. 
Hence by Theorem \ref{alpha+1}, $a= \alpha+1$. 
Since $R_n\ne 0$ if and only if $\deg [nD] \ge 0$, and 
\[ \leqno{(\ref{pg}.1)}\quad [nD] +[(\alpha -n)D ] = [\alpha D] = -2\quad \mbox{for $1\le n \le \alpha -1$}, \]  
we have 
\[\leqno{(\ref{pg}.2)}\quad \deg [nD] = -1 \quad \mbox{for $1\le n \le \alpha -1$.}\]   
In particular, if $\alpha \ge 2$ and $p_g(R)=1$, then $\deg E = -1$.  

Also, since $[(\alpha +1)D] = -2 + \deg E + r$ and $\dim R_{\alpha +1}\le 2$,
we get $r\le 4$.  If $r =4$, then $\alpha \le 2$ by the following Lemma \ref{r=3}. 
So, we assume $r=3$ in the following and assume $q_1\le q_2\le q_3$. \par
Since $a=\alpha +1$ in this case, every element of degree $\le 2 \alpha + 1$ is a 
member of minimal generating set of $R$. Hence we must have  
\[\leqno{(\ref{pg}.3)}\quad \sum_{n=\alpha +1}^{2\alpha +1} \dim R_n \le 3.\]
On the other hand, since $\deg [nD] = -1$ for $1\le n \le \alpha -1$ and 
$(\alpha + n)D \sim K_X + \FRC(D) + nD$   
for every $n$, $1\le n\le \alpha - 1$, we have  
\[\leqno{(\ref{pg}.4)}\quad R_{\alpha +n}\ne 0 \LLR \mbox{no $q_i$  divides } n.\]

Thus, we conclude that every $n, 1\le n \le \alpha-1$, except at most $3$ is a multiple of some 
$q_i$.  This excludes the possibility $\alpha =6$ since, every $q_i$ should be relatively prime 
to $6$ and $n=1,2,3,4$ are not multiple of any $q_i$. \par
Next, we assume that $\alpha$ is even and $\ge 8$.  Then since every $q_i$ should be odd, 
we should have $(a,b,c) = (\alpha +1, \alpha +2, \alpha +4)$ by  (\ref{pg}.4). 
But this contradicts the condition of \ref{bd-of-c} (2) and hence this case does not occur. 
\par
Now, we assume that $\alpha$ is odd and $\alpha\ge 7$. 
To proceed further, we note that if $q_1=2$ and $q_2=3$, then $R_6 \ne 0$.
Hence we  have either 
$R_{\alpha +2} \ne 0$ or $R_{\alpha +3}\ne 0$. 
\par
Thus, we have either 
\[\leqno{(\ref{pg}.5)}\quad (a,b,c;h) = (\alpha +1, \alpha +2, c; c+3\alpha+3)\quad {\rm or}\]
\[\leqno{(\ref{pg}.6)}\quad (a,b,c;h) = (\alpha +1, \alpha +3, c; c+3\alpha+4)\]

Now we check the condition of \ref{bd-of-c} (1). In case (2.8.5), $c$ should divide  
either $3\alpha +3, 2\alpha+2$ or $2\alpha +1$. Then we should have $(a,b,c;h) = 
(\alpha +1, \alpha +2, 2\alpha +1; 5\alpha+4)$ or 
$(\alpha +1, \alpha +2, (3\alpha +3)/2; 9(\alpha+1)/2+1)$.  
in the former case,  $\alpha +2$ must divide 
either $5\alpha+4, 3\alpha+3$ or $4\alpha+3$ and we can see such cases do not occur. in the latter case, $\alpha+2$ should divide $9(\alpha+1)/2)$, the only solution is 
 $\alpha =7$, then $(a,b,c;h)= (8, 9,c; c+24)$ and only possibility is and

\[(a,b,c; h) = (8,9,12; 36). \quad D = \dfrac{2}{3} P_1 + \dfrac{1}{4} P_2 + \dfrac{1}{8}P_3 -Q .\]

In case (2.8.6), $c$ should divide either $3\alpha +4, 2\alpha+1$ or $2\alpha +3$. 
Then we should have $(a,b,c;h) = 
(\alpha +1, \alpha +3, 2\alpha +1; 5\alpha+5), 
(\alpha +1, \alpha +3, 2\alpha +3; 5\alpha+7)$ or 
$(\alpha +1, \alpha +3, 3\alpha +4; 6\alpha+8)$.  

As in the previous case, since we have assumed $\alpha\ge 7$, the only solutions are  

\[ (8,10,15; 40)\quad D = \dfrac{1}{2}P_1 + \dfrac{2}{5} P_2 + \dfrac{2}{15}P_3 -Q,\]   
\[ (8,10,25; 50)\quad D = \dfrac{1}{2}P_1 + \dfrac{2}{5} P_2 + \dfrac{1}{8}P_3 -Q.\]     
\end{proof}

\begin{lem}\label{r=3} Let $R=R(X,D)$ be as in Theorem \ref{pg}, with $g=0$ and $r=4$.
Then if $p_g(R)\le 1$, then $\alpha \le 2$.
\end{lem}
\begin{proof} If $\alpha \ge 3$, then by  formula (\ref{pg}.2), 
$\deg E=-1$ and $\deg [2D] =-1$. Hence the number of $i$ with $q_i = 2$ is at most $1$ 
and $\deg \alpha D \ge \dfrac{1}{2} + 3\cdot \dfrac{2}{3} = \dfrac{5}{2}$. 
On the other hand, since the type of $R$ is $(\alpha +1,\alpha +1, \alpha +2: 4\alpha +4)$,
$\deg D= \dfrac{4}{(\alpha +1)(\alpha +2)}$, which is smaller than $\dfrac{5}{4\alpha}$.
A contradiction !   
\end{proof}

\section{The classification of the hypersurfaces with $a(R)\le 6$.} 

Now let us begin the classification of normal graded surface 
$R=k[x,y,z]/(f)$ with $a(R)=\alpha \le 6$.
\vskip 0.8cm
\centerline{\large\bf The case  $\alpha=1$}
\vskip 0.8cm

Henceforce, we put $D=K_X +\sum_{i=1}^r \dfrac{q_i-1}{q_i} P_i.$
From \ref{Fund} (3), the maximal value of $\deg D$ is taken when 
$a=b=c=1$ and  $\deg D=4$ in that case. 

\vskip 0.5cm
{\bf Case 1 - A. The case $g>0$.} 
\vskip 0.3cm
Assume that $g\ge 1$. Since $\deg(D)\le 4$, and $\deg D\ge \deg K_X=2g-2$, 
$g\le 3$ and if $g=3$, $D=K_X$. We list the cases by giving the form of 
$D$ and $(a,b,c;h)$. We can easily deduce the general form of the equation 
$f$ from this data. Also, if $f$ with the given weight has an isolated 
singularity, then $k[u,v,w]/(f)\cong R(X,D)$, where  $D$ is a divisor of 
given form. 
\vskip 0.2cm
{\bf (1-A-1)} \quad $g=3,D=K_X;\quad (1,1,1;4).$
\vskip 0.2cm
Next, consider the case $g=2$.  Note that $\dim R_1=\dim H^0(K_X)=g=2$, 
we have $a=b=1$ and $\deg D= 1+\dfrac{3}{c}\le \dfrac{5}{2}\quad (c\ge2)$. 
 Since, either 
$\deg D=2, D=K_X$ or $\deg D\ge \dfrac{5}{2}$, we have 2 cases.
\vskip 0.2cm
{\bf (1-A-2)} \quad $g=2,D=K_X;\quad (1,1,3;6).$

{\bf (1-A-3)} \quad $g=2,D=K_X+\dfrac{1}{2}P;\quad (1,1,2;5).$
\vskip 0.2cm
Next, assume $g=1$.  In this case, $a=1$, $2\le b\le c$ and the maximal 
value of $\deg D$ is $\dfrac{3}{2}$.  Since on the other hand, 
$\deg D\ge \dfrac{r}{2}$ and thus $r\le 3$ and if $r=3$, 
$D=\dfrac{1}{2}(P_1+P_2+P_3)$. 
\vskip 0.2cm
{\bf (1-A-4)} \quad $g=1,D=\dfrac{1}{2}(P_1+P_2+P_3);\quad (1,2,2;6).$
\vskip 0.2cm 
Also, since $\dim R_2=r$, if $r=2$, then $a=1,b=2,c\ge 3$, $\deg(D)\le 
\dfrac{7}{6}$. 
\vskip 0.2cm
{\bf (1-A-5)} \quad $g=1,D=\dfrac{1}{2}(P_1+P_2);\quad (1,2,4;8).$

{\bf (1-A-6)} \quad $g=1,D=\dfrac{1}{2}P_1+\dfrac{2}{3} P_2;\quad (1,2,3;7).$
\vskip 0.2cm

If $g=1$ and $D=\dfrac{q-1}q P$, we have $q-1$ new generators in degrees 
$1,3,\ldots, q$. Hence $q\le 4$.  
\vskip 0.2cm
{\bf (1-A-7)} \quad $g=1,D=\dfrac{1}{2}P;\quad (1,4,6;12).$

{\bf (1-A-8)} \quad $g=1,D=\dfrac{2}{3}P;\quad (1,3,5;10).$

{\bf (1-A-9)} \quad $g=1,D=\dfrac{3}{4}P;\quad (1,3,4;9).$
\vskip 0.2cm
We have $9$ types when $g\ge 1$. 
\vskip 0.5cm
{\bf Case 1- B. The case $g=0$ and $r\ge 4$.} 
\vskip 0.3cm
 Since $\deg(K_X)=-2$ and $\deg D>0$, 
we have $r\ge 3$.  On the other hand, since $R_1=H^0(K_X)=0$, 
$a\ge 2, c\ge  3$, we have $\deg D\le \dfrac{2}{3}<1$. 
Since $\deg D\ge -2+r/2$, we have $r\le 5$. 
\vskip 0.2cm
Now, since $\deg [ 2D ] =r-4$, $\dim R_2=2,1$, respectively, if 
$r=5,4$.  \par

Thus if $r=5$, then $a=b=2$ and $c\ge 3$. 
Since $3.\dfrac{1}{2}+2.\dfrac{2}{3}-2 = \dfrac{5}{6}> \dfrac{2}{3}$, 
the only possible cases for $(q_1,\ldots, q_5)$ are $(2,2,2,2,2)$ and 
$(2,2,2,2,3)$. 

\vskip 0.2cm

{\bf (1-B-1)} \quad $D=K_X+\dfrac{1}{2}(P_1+P_2+\ldots+P_5);\quad (2,2,5;10).$

{\bf (1-B-2)} \quad $D=K_X+\dfrac{1}{2}(P_1+P_2+P_3+P_4)+\dfrac{2}{3}P_5;
\quad (2,2,3;8).$
\vskip 0.2cm
Henceforce we assume $r=4$ and express $D$ by $(q_1,q_2,q_3,q_4)$ and 
{\bf we always assume} $q_1\le q_2\le q_3\le q_4$. 
In this case, $a=2$ and $3\le b\le c$.  Hence $\deg D\le \dfrac{1}{2}$. 
Since $4.\dfrac{2}{3}-2>\dfrac{1}{2}$, $q_1=2$ and $q_4\ge 3$. \par

Let $s$ be the number of $q_i>2$ ($1\le s\le 3$). Then  
since $\deg [3D] = -6 + 8 -s$, $\dim R_3=0,1,2$ 
 when $s=1,2,3$, respectively. \par

If $s=3$, $\dim R_2+\dim R_3=3$ and we must have $(a,b,c;h)=(2,3,3;9)$. 
\vskip 0.2cm
{\bf (1-B-3)} \quad $D=
K_X+\dfrac{1}{2}P_1+\dfrac{2}{3}(P_2+P_3+P_4);\quad (2,3,3;9).$
\vskip 0.2cm
If $s=2$, $a=2,b=3$ and $c\ge 4$.  Hence $\deg D = \dfrac{1}{6}+\dfrac{1}{c}
\le \dfrac{5}{12}$. Also, since $-2+(\dfrac{1}{2}+\dfrac{1}{2}
+\dfrac{2}{3}+\dfrac{3}{4})=\dfrac{5}{12}$, we have $2$ types.

\vskip 0.2cm
{\bf (1-B-4)} \quad $D=
K_X+\dfrac{1}{2}(P_1+P_2)+\dfrac{2}{3}(P_3+P_4);\quad (2,3,6;12).$

{\bf (1-B-5)} \quad $D=
K_X+\dfrac{1}{2}(P_1+P_2)+\dfrac{2}{3}P_3+\dfrac{3}{4}P_4;
\quad (2,3,4;10).$
\vskip 0.2cm
Now we trat the case $(2,2,2,q)$, $q\ge 3$.  In this case, $R_3=0$ and 
$\dim R_4=1$ or $2$ according to $q=3$ or $q\ge 4$. In the latter case,
$\dim R_5=0$ or $1$ according to $q=4$ or $q\ge 5$. Hence, if $q\ge 5$, 
we have already $3$ generators of $R$. 
\vskip 0.2cm
{\bf (1-B-6)} \quad $D=
K_X+\dfrac{1}{2}(P_1+P_2+P_3)+\dfrac{4}{5}P_4
;\quad (2,4,5;12).$

{\bf (1-B-7)} \quad $D=
K_X+\dfrac{1}{2}(P_1+P_2+P_3)+\dfrac{3}{4}P_4
;\quad (2,4,7;14).$

{\bf (1-B-8)} \quad $D=
K_X+\dfrac{1}{2}(P_1+P_2+P_3)+\dfrac{2}{3}P_4
;\quad (2,6,9;18).$
\vskip 0.2cm
We have 8 types in this case.
\vskip 0.5cm
{\bf Case 1 - C. The case $g=0$ and $r=3$.}
\vskip 0.3cm
We have to determine $(q_1,q_2,q_3)$. In this case, 
$ R_1= R_2=0$ and $\dim R_3=1$ or $0$ according to 
$q_1=2$ or $q_1\ge 3$.  

\vskip 0.2cm
{\bf Case 1. $q_1\ge 3$.}

In this case, $a=3$ and $4\le b\le c$.  Hence $\deg D\le \dfrac{1}{4}$. 
Hence either $q_1=3$ or $q_1=q_2=q_3=4$. 

\vskip 0.2cm
{\bf (1-C-1)} \quad $D=
K_X+\dfrac{3}{4}(P_1+P_2+P_3);\quad (3,4,4;12).$
\vskip 0.2cm
{\bf Henceforce we assume $q_1=3$. }

$R_4\ne 0$ if and only if $q_2\ge 4$.  In this case, 
$a=3,b=4$ and $c\ge 5$.  Hence $\deg D\le \dfrac{13}{60}
=\dfrac{2}{3}+\dfrac{3}{4}+\dfrac{4}{5}-2$. Hence we have only 2 
possibilities;
\vskip 0.2cm
{\bf (1-C-2)} \quad $D=
K_X+\dfrac{2}{3}P_1+\dfrac{3}{4}P_2+\dfrac{4}{5}P_3;\quad (3,4,5;13).$

{\bf (1-C-3)} \quad $D=
K_X+\dfrac{2}{3}P_1+\dfrac{3}{4}(P_2+P_3); \quad (3,4,8;16).$
\vskip 0.2cm

Next, assume $q_1=q_2=3$.  Hence $\deg D= \dfrac{q_3-1}{q_3}-\dfrac{2}{3}$.
On the other hand, since 
 $R_4=0$,  $a=3$, $b\ge 5$ and $c\ge 6$ and $\deg D \le \dfrac{1}{6}$. 
This implies $q_3\le 6$. 

\vskip 0.2cm
{\bf (1-C-4)} \quad $D=
K_X+\dfrac{2}{3}(P_1+P_2)+\dfrac{5}{6}P_3;\quad (3,5,6;15).$

{\bf (1-C-5)} \quad $D=
K_X+\dfrac{2}{3}(P_1+P_2)+\dfrac{4}{5}P_3;\quad (3,5,9;18).$

{\bf (1-C-6)} \quad $D=
K_X+\dfrac{2}{3}(P_1+P_2)+\dfrac{3}{4}P_3;\quad (3,8,12;24).$
\vskip 0.2cm

This completes the case $q_1=3$.

\vskip 0.2cm

{\bf Case 2. $q_1=2$.}

\vskip 0.2cm
In this case, $a\ge 4$ and $R_4\ne 0$ if and only if $q_2\ge 4$. 

\vskip 0.2cm
{\bf First, we consider the case $q_1=2$ and $q_2=3$ ($q_3\ge 7$). }

In this case, $\deg [4D]=-1=\deg [5D]=\deg [7D], \deg [6D]=0$. Hence 
$a=6$ and $b\ge 8$. Hence $\deg D\le \dfrac{1}{18}=\dfrac{8}{9}-\dfrac{5}{6}$.   This shows that $7\le q_3\le 9$ and actually these cases gives the 
hypersurfaces.  
\vskip 0.2cm
{\bf (1-C-7)} \quad $D=
K_X+\dfrac{1}{2}P_1+\dfrac{2}{3}P_2+\dfrac{6}{7}P_3;\quad (6,14,21;42).$

{\bf (1-C-8)} \quad $D=
K_X+\dfrac{1}{2}P_1+\dfrac{2}{3}P_2+\dfrac{7}{8}P_3;\quad (6,8,15;30).$

{\bf (1-C-9)} \quad $D=
K_X+\dfrac{1}{2}P_1+\dfrac{2}{3}P_2+\dfrac{8}{9}P_3;\quad (6,8,9;24).$

\vskip 0.2cm
{\bf Next, we consider the case $q_1=2$ and $q_2\ge 4$. }

In this case, $\deg [4D]=0$ and $a=4,b\ge 5, c\ge 6$. 
Hence  $\deg D\le \dfrac{2}{15}=(\dfrac{1}{2}+\dfrac{4}{5}+\dfrac{5}{6})-2$.
Hence $q_2\le 5$ and if $q_2=5$, the possibility is the following 2 cases.

\vskip 0.2cm
{\bf (1-C-10)} \quad $D=
K_X+\dfrac{1}{2}P_1+\dfrac{4}{5}P_2+\dfrac{5}{6}P_3;\quad (4,5,6;16).$

{\bf (1-C-11)} \quad $D=
K_X+\dfrac{1}{2}P_1+\dfrac{4}{5}(P_2+P_3);\quad (4,5,10;20).$

\vskip 0.2cm
{\bf The remaining case is $q_1=2,q_2=4\; (q_3\ge 5)$. }

Since $\dim R_4=1$ and $R_5=0$ and hence $a=4,b\ge 6,c\ge 7$ and 
$\deg D= \dfrac{q_3-1}{q_3}-\dfrac{3}{4}\le \dfrac{3}{28}$.  Hence 
$5\le q_3\le 7$ and actually these cases give hypersurfaces. 

This finishes the classification !
\vskip 0.2cm
{\bf (1-C-12)} \quad $D=
K_X+\dfrac{1}{2}P_1+\dfrac{3}{4}P_2+\dfrac{4}{5}P_3;\quad (4,10,15;30).$

{\bf (1-C-13)} \quad $D=
K_X+\dfrac{1}{2}P_1+\dfrac{3}{4}P_2+\dfrac{5}{6}P_3;\quad (4,6,11;22).$

{\bf (1-C-14)} \quad $D=
K_X+\dfrac{1}{2}P_1+\dfrac{3}{4}P_2+\dfrac{6}{7}P_3;\quad (4,6,7;18).$

\vskip 0.8cm
\centerline{\large\bf The case  $\alpha=2$}
\vskip 0.8cm

We   assume that $R=R(X,D)\cong k[u,v,w]/(f)$ 
with 
$$\deg(u,v,w;f)=(a,b,c;h); \quad h=a+b+c+2.$$

We always assume $(a,b,c)=1$ and $a\le b\le c$.  Since $R$ is Gorenstein with $a(R)=2$, 
$2D$ is linearly equivalent to $K_X+\FRC(D)$.  Hence we may assume that

$$D=E+\sum_{i=1}^r \dfrac{q_i-1}{2q_i}P_i,$$
 where $2E \sim K_X$ and every $q_i$ is odd.  

We divide the cases according to (A) $a=b=1$, 
(B) $a=1,b\ge 2$, and (C) $a>1$.    
\vskip 0.3cm
{\bf Case 2 - A.} When  $a=b=1$, we have the following $4$ types.  
\vskip 0.2cm
\noindent  
{\bf (2-A-1)} \quad $g=6, D=E$ with $2E\sim K_X;   
\quad (1,1,1;5).$
\vskip 0.2cm
\noindent  
{\bf (2-A-2)} \quad $g=4, D=E$ with $2E\sim K_X;    
\quad (1,1,2;6).$
\vskip 0.2cm
\noindent  
{\bf (2-A-3)} \quad $g=3, D=E+\frac{1}{3}P$ with $2E\sim K_X;   
\quad (1,1,3;7).$
\vskip 0.2cm
\noindent  
{\bf (2-A-4)} \quad $g=3, D=E$ with $2E\sim K_X;   
\quad (1,1,4;8).$
\vskip 0.3cm
{\bf Case 2 - B.}  $a=1$ and $b\ge 2$
\vskip 0.2cm
If $b=2$, we have $g=\dim R_2=2$ and we have the following types.
\vskip 0.2cm
\noindent  
{\bf (2-B-1)} \quad $g=2, D=E+\dfrac{1}{3}P$ with $2E\sim K_X;  
\quad (1,2,3;8).$
\vskip 0.2cm
\noindent  
{\bf (2-B-2)} \quad $g=2, D=E$ with $2E\sim K_X;    
\quad (1,2,5;10).$
\vskip 0.2cm
If $b\ge 3$, then $g=1$.  Since $2E\sim 0$ and $R_1\ne 0$,  
we have $E=0$ and $\deg [3D]=r$. Hence $r\le 3$ and $\deg D
\le \dfrac{9}{1\cdot 3\cdot 3}=1.$ 
If $r\ge 2$, then $a=1$ and $b=3$. We have the following cases.
\vskip 0.2cm
\noindent  
{\bf (2-B-3)} \quad $g=1, D=\dfrac{1}{3}(P_1+P_2+P_3);  
\quad (1,3,3;9).$
\vskip 0.2cm
\noindent  
{\bf (2-B-4)} \quad $g=1, D=\dfrac{1}{3}P_1+\dfrac{2}{5}P_2;  
\quad (1,3,5;11).$
\vskip 0.2cm
\noindent  
{\bf (2-B-5)} \quad $g=1, D=\dfrac{1}{3}(P_1+P_2);  
\quad (1,3,6;12).$
\vskip 0.2cm
If $r=1$, then $b\ge 5$  and 
$\deg D \le \dfrac{15}{1\cdot 5\cdot 7}=\dfrac{3}{7}$. We have the 
following cases. 
\vskip 0.2cm
\noindent  
{\bf (2-B-6)} \quad $g=1, D=\dfrac{3}{7}P;  
\quad (1,5,7;15).$
\vskip 0.2cm
\noindent  
{\bf (2-B-7)} \quad $g=1, D=\dfrac{2}{5}P;  
\quad (1,5,8;16).$
\vskip 0.2cm
\noindent  
{\bf (2-B-8)} \quad $g=1, D=\dfrac{1}{3}P;  
\quad (1,6,9;18).$
\vskip 0.2cm
This finishes the case $a=1,b\ge 2$.
\vskip 0.2cm
{\bf Case 2 - C.}  $a\ge 2$.  
\vskip 0.2cm
This is equivalent to say that $R_1=H^0(X,\OO_X(D))=0$. If this 
is the case, we have 
$$\deg D \le \dfrac{9}{2\cdot 2\cdot 3}=\dfrac{3}{4}<1.$$

Since $\deg D\ge g-1$, $g=0$ or $1$ in this case. 
\vskip 0.2cm
First, assume $g=1$. Then $2E\sim 0$ and $\deg E=0$. Since $\deg [3D]
\ge 1$, $R_3\ne 0$. Hence $a=2$ and $b=3$ in this case.
We have the following cases.
\vskip 0.2cm
{\bf (2-C-1)} \quad $g=1, D=E+\dfrac{1}{3}P$ with $E\ne 0, 2E\sim 0
  ;\quad (2,3,7;14).$
\vskip 0.2cm
{\bf (2-C-2)} \quad $g=1, D=E+\dfrac{2}{5}P$ with $E\ne 0, 2E\sim 0  ;\quad (2,3,5;12).$
\vskip 0.2cm
Next, assume $g=0$.  Then since $2E\sim K_X$, $\deg E=-1$ and 
$\deg [3D]=-3+r\ge 0$, which implies $a=3$.  Since $\dim R_3\le 2$, 
$r=3$ or $4$. \par
If $\dim R_3=2$ and $r=4$, we have only one case. 
\vskip 0.2cm
{\bf (2-C-3)} \quad $g=0, D=E+\dfrac{1}{3}(P_1+\ldots +P_4)$ with 
 $\deg E=-1 ;\quad (3,3,4;12).$
\vskip 0.2cm
If $r=3$, then $a=3, b\ge 5$ and $\deg D\le 
\dfrac{15}{3\cdot 5\cdot 5}=\dfrac{1}{5}$. We have the following cases.
\vskip 0.2cm
{\bf (2-C-4)} \quad $g=0, D=E+\dfrac{2}{5}(P_1+P_2 +P_3)$ with 
 $\deg E=-1 ;\quad (3,5,5;15).$
\vskip 0.2cm
{\bf (2-C-5)} \quad $g=0, D=E+\dfrac{1}{3}P_1+\dfrac{2}{5}P_2 +
\dfrac{3}{7}P_3$ with 
 $\deg E=-1 ;\quad (3,5,7;17).$
\vskip 0.2cm
{\bf (2-C-6)} \quad $g=0, D=E+\dfrac{1}{3}P_1+\dfrac{2}{5}(P_2 +
P_3)$ with  $\deg E=-1 ;\quad (3,5,10;20).$
\vskip 0.2cm
{\bf (2-C-7)} \quad $g=0, D=E+\dfrac{1}{3}(P_1+P_2) +\dfrac{4}{9}P_3$
 with  $\deg E=-1 ;\quad (3,7,9;21).$
\vskip 0.2cm
{\bf (2-C-8)} \quad $g=0, D=E+\dfrac{1}{3}(P_1+P_2) +\dfrac{3}{7}P_3$
 with  $\deg E=-1 ;\quad (3,7,12;24).$
\vskip 0.2cm
{\bf (2-C-9)} \quad $g=0, D=E+\dfrac{1}{3}(P_1+P_2) +\dfrac{2}{5}P_3$
 with  $\deg E=-1 ;\quad (3,10,15;30).$
\vskip 0.2cm
This finishes the case $a(R)=2$.

\vskip 0.8cm
\centerline{\large\bf The case  $\alpha=3$}
\vskip 0.8cm

We  assume that $R=R(X,D)\cong k[u,v,w]/(f)$ 
with 
$$\deg(u,v,w;f)=(a,b,c;h); \quad h=a+b+c+3.$$

 Since $R$ is Gorenstein with $a(R)=3$, 
$3D$ is linearly equivalent to $K_X+\FRC(D)$.  By \ref{Fund} (1),
 we may assume that

$$\leqno{(3.1.1)}\quad 
D=E+\sum_{i=1}^r{}_{q_i \equiv 1 (\mod 3)} \dfrac{q_i-1}{3q_i}P_i+
\sum_{i=1}^s {}_{q_i \equiv 2 (\mod 3)}\dfrac{2q_i-1}{3q_i}Q_i.$$

By \ref{Fund} (1), we have 
$$\leqno{(3.1.2)} \quad 
3E + \sum_{i=1}^s Q_i \sim K_X; \quad \deg E=\dfrac{2g-2-s}{3}.$$

Since $\dfrac{q_i-1}{3q_i}\ge \dfrac{1}{4}$ if $q_i \equiv 1 (\mod 3)$, 
$q_i>1$  and $\dfrac{2q_i-1}{3q_i}\ge \dfrac{1}{2}$ if $q_i \equiv 2 (\mod 3)$,
we have 
$$\leqno{(3.1.3)} \quad 
\deg D\ge \deg E+\dfrac{r}{4}+\dfrac{s}{2}=\dfrac{2g-2}{3}+
\dfrac{r}{4}+\dfrac{s}{6}.$$

\vskip 0.3cm
We divide the cases according to (A) $a=b=1$, 
(B) $a=1,b\ge 2$, (C) $a\ge 2$. 
\vskip 0.5cm
{\bf Case 3 - A.  $a=b=1$.}
\vskip 0.5cm
In this case, the type of $R$ is the form $(1,1,c; c+5)$. 
By Lemma \ref{bd-of-c}, $c\le 5$ and $c\ne 3$.   
So we have for cases $c=1,2,4,5$. We can calculate the genus $g$ by 
$g=\dim R_3$. 
\vskip 0.2cm\noindent
{\bf (3-A-1)} \quad $g=10$, $D=E$ with $\deg E=6$, $3E\sim K_X$ 
;\quad $(1,1,1;4).$ $X$ is not hyperelliptic.
\vskip 0.2cm\noindent
{\bf (3-A-2)} \quad $g=6,D=E+\dfrac{1}{2}Q$ with $\deg E =3, 3E+Q\sim K_X$; 
\quad $(1,1,2;7)$. \par
Since $g=6$ and $\deg D=\dfrac{7}{2}=\dfrac{10}{3}+\dfrac{1}{6}$, this is 
the only possible equation for $D$.
\vskip 0.2cm\noindent
{\bf (3-A-3)} \quad $g=4,D=E+\dfrac{1}{4}P$ with $\deg E=2, 3E\sim K_X$;\quad $(1,1,4;9).$
\vskip 0.2cm\noindent
{\bf (3-A-4)} \quad $g=4,D=E$ with $\deg E=2$, $3E\sim K_X;\quad (1,1,5;10).$
\vskip 0.5cm
{\bf Case 3 - B.  $a=1,b\ge 2$.}
\vskip 0.5cm
In this case, since $R_1=H^0(X,O_X(E))\ne 0$, we may assume $E\ge 0$.
Since $g=\dim R_3$, we have $g\le 3$. Also, 
$g\ge 2$ if and only if $b\le 3$ and $g=1$ otherwise. 
\par
If $g =3$, then $c\le 3$ and we have the following cases.
\vskip 0.2cm\noindent
{\bf (3-B-1)} \quad $g=3,D=\dfrac{1}{2}\sum_{i=1}^4Q_i$ with 
$\sum_{i=1}^4Q_i\sim K_X$;\quad $(1,2,2;8).$
\vskip 0.2cm\noindent
{\bf (3-B-2)} \quad $g=3,D=E+\dfrac{1}{2}Q$ with 
$\deg E=1$, $3E+Q\sim K_X;\quad (1,2,3;9).$
\vskip 0.2cm
If $g=2$, then $b\le 3$ and $c\ge 4$. Aso, by (3.1.2), 
$E=0, s=2$ and $Q_1+Q_2\sim K_X$.  Hence   
$\dim R_2 = h^0(K_X)=2$. Thus we have $b=2$. \par
Since $E=0$, $\deg [4D] = r + 2s \ge 4$ and $c=4$ if and only if $r>0$. 
\vskip 0.2cm\noindent
{\bf (3-B-3)} \quad $g=2,D=\dfrac{1}{4}P+\dfrac{1}{2}(Q_1+Q_2)$ 
with $Q_1+Q_2\sim K_X;
\quad (1,2,4;10).$\par
Now, since $E=0$ and $s=2$, we have $\deg D\ge 1$. 
On the other hand, since $a=1, b=2$, $deg D= \dfrac{c+6}{2c}$ and $\deg D\ge 1$ if and only if 
$c\le 6$.  Hence we have the following cases. 
\vskip 0.2cm\noindent
{\bf (3-B-4)} \quad $g=2,D=\dfrac{3}{5}Q_1+\dfrac{1}{2}Q_2;$ with $Q_1+Q_2
\sim K_X;\quad (1,2,5;11).$
\vskip 0.2cm\noindent
{\bf (3-B-5)} \quad $g=2,D=\dfrac{1}{2}(Q_1+Q_2)$ 
with $Q_1+Q_2\sim K_X;\quad (1,2,6;12).$
\vskip 0.2cm
In the remaining cases, $a=1$ and $b\ge 4$.  Then we have $g=1$ and by (3.1.2) 
we have $E=0$ and $s=0$. Also, since $\deg [4D] =r$, $b=4$ if and only if $r\ge 2$. 
If this is the case, we have the following $3$ cases.  
\vskip 0.2cm\noindent
{\bf (3-B-6)} \quad $g=1,D=\dfrac{1}{4}(P_1+P_2+P_3)
  ;\quad (1,4,4;12).$
\vskip 0.2cm\noindent
{\bf (3-B-7)} \quad $g=1,D=\dfrac{1}{4}P_1+\dfrac{2}{7}P_2
  ;\quad (1,4,7;15).$
  \vskip 0.2cm\noindent
{\bf (3-B-8)} \quad $g=1,D=\dfrac{1}{4}(P_1+P_2)
  ;\quad (1,4,8;16).$
\par
If $g=1, s=0$ and $r=1$, $b\ge 7$ and $c\ge 10$. Hence $\deg D
 \le \dfrac{21}{70}=\dfrac{3}{10}$.  
We have the following $3$ types.
\vskip 0.2cm\noindent
{\bf (3-B-9)} \quad $g=1,D=\dfrac{1}{4}P 
  ;\quad (1,8,12;24).$
\vskip 0.2cm\noindent
{\bf (3-B-10)} \quad $g=1,D=\dfrac{2}{7}P 
  ;\quad (1,7,11;22).$
\vskip 0.2cm\noindent
{\bf (3-B-11)} \quad $g=1,D=\dfrac{3}{10}P
  ;\quad (1,7,10;21).$
\vskip 0.5cm
{\bf Case 3 - C.  $a\ge 2$.}
\vskip 0.5cm
In this case, $\deg D \le \dfrac{10}{2\cdot 2\cdot 3}<1$.  \par
Since $g=\dim R_3$, $g\le 2$.  But since the hypersurfaces of type $(3,3,c; c+9)$ cannot 
be normal, $g\le 1$.  \par
If $g=1$, by (3.1.2), either $\deg E=0 =s$ or $\deg E= -1$ and $s=3$.  In the first case,
since $3E \sim 0, R_1=0, E\not\sim 0$ and  $R_2 =0$. Hence $\deg D \le  
\dfrac{15}{3\cdot 4\cdot 5}=\dfrac{1}{4}$ and we have the following case. 
\vskip 0.2cm\noindent
{\bf (3-C-1)} \quad $g=1, D=E+\dfrac{1}{4}P$, $3E\sim 0,E\not\sim 0 
  ;\quad (3,4,5;15).$
\vskip 0.2cm
In the latter case, since $\deg [2D] \ge 1$, we have $a=2, b=3$.  
Since $\deg D \le \dfrac{12}{2\cdot 3\cdot 4}= \dfrac{1}{2}$, the 
following case is the only possiblity.   
\vskip 0.2cm\noindent
{\bf (3-C-2)} \quad $g=1, D=E+\dfrac{1}{2}(P_1+P_2+P_3)$, $3E+P_1+P_2+P_3\sim 0 
 ;\quad (2,3,4;12).$
\vskip 0.2cm
Now, until the end of the case $\alpha =2$, we assume $g=0$. 
By (3.1.2), we have $\deg E = - \dfrac{s+2}{3}$ and $\deg [2D] = s + 2\deg E = 
\dfrac{s-4}{3}$.  \par
If $s=7$, then $a=b=2$ and $\deg D \ge \dfrac{1}{2}$, which implies $c\le 7$ and we have 
the following $2$ cases.    

\vskip 0.2cm\noindent
{\bf (3-C-3)} \quad $g=0, D=E+\dfrac{1}{2}(P_1+\ldots +P_6)+\dfrac{3}{5}P_7, \deg E=-3
  ;\quad (2,2,5;12).$
\vskip 0.2cm\noindent
{\bf (3-C-4)} \quad $g=0, D=E+\dfrac{1}{2}(P_1+\ldots +P_6+P_7), \deg E=-3
  ;\quad (2,2,7;14).$
\vskip 0.2cm
If $s=4$ and $\deg E = -2$, we have $a=2$ and since $\deg [4D] = r \le 1$,  
 $b=4$ if and only if $r=1$, otherwise, $b\ge 5$. 
 If $r=1$, then $\deg D \ge \dfrac{1}{4}$ and $a=2,b=4$ and $c\le 9$ and 
 we have the following $3$ cases. 
 \vskip 0.2cm\noindent
{\bf (3-C-5)} \quad $g=0, D=E+\dfrac{1}{4}P+\dfrac{3}{5}Q_1+\dfrac{1}{2}
(Q_2+Q_3+Q_4)$, $\deg E=-2
  ;\quad (2,4,5;14).$
\vskip 0.2cm\noindent
{\bf (3-C-6)} \quad $g=0, D=E+\dfrac{2}{7}P+\dfrac{1}{2}
(Q_1+\ldots+Q_4)$, $\deg E=-2
  ;\quad (2,4,7;16).$
\vskip 0.2cm\noindent
{\bf (3-C-7)} \quad$g=0, D=E+\dfrac{1}{4}P+\dfrac{1}{2}
(Q_1+\ldots+Q_4)$, $\deg E=-2
  ;\quad (2,4,9;18).$
\vskip 0.2cm
If $s=4$ and $r=0$, then $b\ge 5$ and we have $\dfrac{1}{10}\le \deg D \le 
\dfrac{15}{2\cdot 5\cdot 5}$, we have the following $6$ cases. 
\vskip 0.2cm\noindent
{\bf (3-C-8)} \quad $g=0, D=E+\dfrac{1}{2}Q_1+\dfrac{3}{5}(Q_2+Q_3+Q_4)$, $\deg E=-2
  ;\quad (2,5,5;15).$
\vskip 0.2cm\noindent
{\bf (3-C-9)} \quad $g=0, D=E+\dfrac{5}{8}Q_1+\dfrac{3}{5}Q_2+\dfrac{1}{2}
(Q_3+Q_4)$, $\deg E=-2
  ;\quad (2,5,8;18).$
\vskip 0.2cm\noindent
{\bf (3-C-10)} \quad $g=0, D=E+\dfrac{3}{5}(Q_1+Q_2)+\dfrac{1}{2}
(Q_3+Q_4)$, $\deg E=-2
  ;\quad (2,5,10;20).$
\vskip 0.2cm\noindent
{\bf (3-C-11)} \quad $g=0, D=E+\dfrac{7}{11}Q_1+\dfrac{1}{2}
(Q_2+Q_3+Q_4)$, $\deg E=-2
  ;\quad (2,8,11;24).$
\vskip 0.2cm\noindent
{\bf (3-C-12)} \quad $g=0, D=E+\dfrac{5}{8}Q_1+\dfrac{1}{2}
(Q_2+Q_3+Q_4)$, $\deg E=-2
  ;\quad (2,8,13;26).$
\vskip 0.2cm\noindent
{\bf (3-C-13)} \quad $g=0, D=E+\dfrac{3}{5}Q_1+\dfrac{1}{2}
(Q_2+Q_3+Q_4)$, $\deg E=-2
  ;\quad (2,10,15;30).$
 \vskip 0.2cm
If $s=1$,  since $\deg (K_X+\FRC(D))>0$, we must have $r\ge 2$. On the other 
hand, since $R_2=R_3=0$ and $\deg [4D] = r-2\le 1$, $r\le 3$.  
Hence $a\ge 4$ and $\deg D\le \dfrac{16}{4\cdot 4\cdot 5}=\dfrac{1}{5}$.
 If $r\ge 3$, then $\deg D\ge \dfrac{1}{4}$.  Hence $r=2$, $a=4,b\ge 5$.
 We have the following $7$ cases.  
\vskip 0.2cm\noindent
{\bf (3-C-14)} \quad $g=0, D=E+\dfrac{3}{5}Q+\dfrac{2}{7}P_1+\dfrac{1}{4}P_2$, 
$\deg E=-1
  ;\quad (4,5,7;19).$
\vskip 0.2cm\noindent
{\bf (3-C-15)} \quad $g=0, D=E+\dfrac{5}{8}Q+\dfrac{1}{4}(P_1+P_2)$, 
$\deg E=-1
  ;\quad (4,5,8;20).$
\vskip 0.2cm\noindent
{\bf (3-C-16)} \quad $g=0, D=E+\dfrac{3}{5}Q+\dfrac{1}{4}(P_1+P_2)$,
$\deg E=-1
  ;\quad (4,5,12;24).$
\vskip 0.2cm\noindent
{\bf (3-C-17)} \quad $g=0, D=E+\dfrac{1}{2}Q_1+\dfrac{2}{7}P_1+\dfrac{3}{10}P_2$, 
$\deg E=-1  ;\quad (4,7,10;24).$
\vskip 0.2cm\noindent
{\bf (3-C-18)} \quad $g=0, D=E+\dfrac{1}{2}Q_1+\dfrac{2}{7}(P_1+P_2)$, 
$\deg E=-1  ;\quad (4,7,14;28).$
\vskip 0.2cm\noindent
{\bf (3-C-19)} \quad $g=0, D=E+\dfrac{1}{2}Q_1+\dfrac{1}{4}P_1+\dfrac{3}{10}P_2$, 
$\deg E=-1  ;\quad (4,10.17;34).$
\vskip 0.2cm\noindent
{\bf (3-C-20)} \quad $g=0, D=E+\dfrac{1}{2}Q_1+\dfrac{1}{4}P_1+\dfrac{2}{7}P_2$, 
$\deg E=-1  ;\quad (4,14,21;42).$
\vskip 0.2cm
This finishes the case $a(R)=3$.

\vskip 0.8cm
\centerline{\large\bf The case  $\alpha=4$}
\vskip 0.8cm

We   assume that $R=R(X,D)\cong k[u,v,w]/(f)$ with 
$$\deg(u,v,w;f)=(a,b,c;h); \quad h=a+b+c+4.$$

Since $4D \sim K_X+\FRC(D)$,   we may assume that

$$\leqno{(4.1.1)}\quad 
D=E+\sum_{i=1}^r{}_{q_i \equiv 1 (\mod 4)} \dfrac{q_i-1}{4q_i}P_i+
\sum_{i=1}^s {}_{q_i \equiv 3 (\mod 4)}\dfrac{3q_i-1}{4q_i}Q_i,$$
where $E$ is an integral divisor on $X$. 

Since $4D\sim K_X+\FRC(D)$, we have 
$$\leqno{(4.1.2)} \quad 
4E + 2 \sum_{i=1}^s Q_i \sim K_X; \quad \deg E=\dfrac{g-1-s}{2}.$$

Since $\dfrac{q_i-1}{4q_i}\ge \dfrac{1}{5}$ if $q_i \equiv 1 (\mod 4)$, 
$q_i>1$  and $\dfrac{3q_i-1}{4q_i}\ge \dfrac{2}{3}$ if $q_i \equiv 3 (\mod 4)$,
we have 
$$\leqno{(4.1.3)}\quad 
\deg D\ge \deg E+\dfrac{r}{5}+\dfrac{2s}{3}=\dfrac{g-1}{2}+
\dfrac{r}{5}+\dfrac{s}{6}.$$

\vskip 0.3cm
We divide the cases according to (A) $a=b=1$, 
(B) $a=1,b\ge 2$, (C) $a\ge 2$. 

\vskip 0.5cm
{\bf Case A.  $a=b=1$.}
\vskip 0.5cm
In this case, the type of $R$ is the form $(1,1,c; c+6)$. 
By \ref{bd-of-c}, $c\le 6$ and we have the following cases.
\vskip 0.2cm\noindent
{\bf (4-A-1)} \quad $g=15$, $D=E$ with $\deg E=6$, $3E\sim K_X$ 
;\quad $(1,1,1;4).$ $X$ is not hyperelliptic.
\vskip 0.2cm\noindent
{\bf (4-A-2)} \quad $g=9,D=E$ with $\deg E=4, 4E\sim K_X$; 
\quad $(1,1,2;8)$. \par
\vskip 0.2cm\noindent
{\bf (4-A-3)} \quad $g=7,D=E$ with $\deg D=3, 3E\sim K_X ;\quad (1,1,3;9).$
\vskip 0.2cm\noindent
{\bf (4-A-4)} \quad $g=5,D=E+\dfrac{1}{5}$ with $\deg E=2$, $4E\sim K_X;
\quad (1,1,4;10).$
\vskip 0.2cm\noindent
{\bf (4-A-5)} \quad $g=5$, hyperelliptic, $D=E$ with $\deg D=2$, $4D\sim K_X;
\quad (1,1,6;12).$
\vskip 0.5cm
{\bf Case B.  $a=1,b\ge 2$.}
\vskip 0.5cm
In this case, since $R_1=H^0(X,O_X(E))\ne 0$, we may assume $E\ge 0$.
Since $g=\dim R_4$, $g\ge 2$ if and only if $b\le 4$ and $g=1$ otherwise. 
Checking the cases of type $(1,2,c; c+7)$ and $(1,3,c;c+8)$, 
we have the following types. 
\vskip 0.2cm\noindent
{\bf (4-B-1)} \quad $g=4,D=E+\dfrac{2}{3}(Q$ with $4E+2Q\sim K_X$; 
\quad $(1,2,3;10)$. 
\vskip 0.2cm\noindent
{\bf (4-B-2)} \quad $g=3,D=E+\dfrac{1}{5}Q$ with $4E\sim K_X$; 
\quad $(1,2,5;12)$. 
\vskip 0.2cm\noindent
{\bf (4-B-3)} \quad $g=3$, hyperelliptic; $D=E$ with $4E\sim K_X$; 
\quad $(1,2,7;14)$. 
\vskip 0.2cm\noindent
{\bf (4-B-4)} \quad $g=3,D=E$ with $4E\sim K_X$; \quad $(1,3,4;12)$.
\vskip 0.2cm\noindent
{\bf (4-B-5)} \quad $g=2,D=\dfrac{1}{5}P+\dfrac{2}{3}Q$ with $2Q\sim K_X$; 
\quad $(1,3,5;13)$. 
\vskip 0.2cm\noindent
{\bf (4-B-6)} \quad $g=2,D=\dfrac{5}{7}Q$ with $2Q\sim K_X$; 
\quad $(1,3,7;15)$. 
\vskip 0.2cm\noindent
{\bf (4-B-7)} \quad $g=2,D=\dfrac{2}{3}Q$ with $2Q\sim K_X$; 
\quad $(1,3,8;16)$. 
\vskip 0.2cm
These finishes the case $a=1$ and $2\le b\le 3$. 
We can also check that the type $(1,4,c;c+9)$ can not give a normal ring. 
Hence we may assume $b\ge 5$ and hence $g=\dim R_4=1$. 
Also, we have  $\deg D \le \dfrac{15}{5\cdot 5}=\dfrac{3}{5}$. Hence 
$E=0$ and $s=0$.  Also, $b=5$  if  and only if $r\ge 2$.  Since 
$\dim R_5\le 3$, we have also $r\le 3$ and if $r=3$, then $\deg D\ge 
\dfrac{3}{5}$. 
\vskip 0.2cm\noindent
{\bf (4-B-8)} \quad $g=1,D=\dfrac{1}{5}(P_1+P_2+P_3); \quad (1,5,5;15)$.
\vskip 0.2cm
If $r=2$, $b=5$ and $c\ge 9$.  Hence $\deg D \le \dfrac{16}{45}=\dfrac{1}{5}+
\dfrac{2}{9}$.
\vskip 0.2cm\noindent
{\bf (4-B-9)} \quad $g=1,D=\dfrac{1}{5}P_1+\dfrac{2}{9}P_2; \quad (1,5,9;19)$. 
\vskip 0.2cm\noindent
{\bf (4-B-10)} \quad $g=1,D=\dfrac{1}{5}(P_1+P_2); \quad (1,5,10;20)$. 
\vskip 0.2cm
If $r=1$, $b\ge 9$ and $c\ge 13$.  Hence $\deg D \le \dfrac{27}{9\cdot 13}=
\dfrac{3}{13}$. Thus we have only the following possibilities.
\vskip 0.2cm\noindent
{\bf (4-B-11)} \quad $g=1,D=\dfrac{1}{5}P; \quad (1,10,15;30)$. 
\vskip 0.2cm\noindent
{\bf (4-B-12)} \quad $g=1,D=\dfrac{2}{9}P; \quad (1,9,14;28)$. 
\vskip 0.2cm\noindent
{\bf (4-B-13)} \quad $g=1,D=\dfrac{3}{13}P; \quad (1,9,13;27)$. 
\vskip 0.2cm\noindent
This finishes the case $R_1\ne 0$.
\vskip 0.5cm
{\bf Case C.  $a\ge 2$.}
\vskip 0.5cm
Since $\deg D\le \dfrac{11}{2\cdot 2\cdot 3}<1$, we have $g\le 2$ by 
(4.1.3).  But if $\dim R_4\ge 2$, $2$ among $a,b,c$ should be even and 
$R$ will not be normal. Hence $g\le 1$. 
\par
If $g=1$, then $s$ is even and $\deg E=-\dfrac{s}{2}$ by (4.1.2).
Also by (4.1.3), $s\le 4$ and we have always $\deg [2D] =0$. 
Hence $\deg D \le \dfrac{12}{2\cdot 3\cdot 3}=\dfrac{2}{3}$. 
If $s=4$, $\deg D \ge \dfrac{2}{3}$ and we have equality.
\vskip 0.2cm\noindent
{\bf (4-C-1)} \quad $g=1,D=E+\dfrac{2}{3}(P_1+P_2+P_3+P_4),
\deg E =-2, -2E\sim P_1+P_2+P_3+P_4; \quad 
(2,3,3;12)$. 
\vskip 0.2cm
If $g=1, \deg E=-1$ and $s=2$, $\deg D\ge \dfrac{1}{3}$ by (5.3.1). 
On the other hand, if $R_2=0$, then $\deg D \le \dfrac{16}{3\cdot 4\cdot 5}$ 
and we have a contradiction.  Thus we must have $\dim R_2=1$, that is, 
$2E +P_1+P_2\sim 0$. If $r\ge 1$, then $\deg D \ge \dfrac{2}{3}+\dfrac{1}{5}$ 
and on the other hand, $\deg D \le \dfrac{16}{3\cdot 4\cdot 5}$ and we have 
a contradiction. 
If $r=0$, the only possibility is the following;  
\vskip 0.2cm\noindent
{\bf (4-C-2)} \quad $g=1, D=E+\dfrac{2}{3}P_1+\dfrac{5}{7}P_2,
\deg E =-1, -2E\sim P_1+P_2; \quad (2,3,7;16)$. 
\vskip 0.2cm
If $g=1, s=0$ and $\deg E=0$, then $\deg D\le \dfrac{16}{2\cdot 5\cdot 5}
<\dfrac{1}{3}$ and we must have $r=1$. 
Since we must have $\dim R_4=1$, the following case is the only 
possible one. 
\vskip 0.2cm\noindent
{\bf (4-C-3)} \quad $g=1, D=E+\dfrac{2}{9}P, E\ne 0, 2E\sim 0; 
\quad (2,5,9;20)$. 
\vskip 0.2cm
Now, we assume $g=0$. Since $R_4=0$, $a\ge 3$ and in 
(4.1.1) we have $\deg  E= \dfrac{-s-1}{2}$ by (4.1.2).  
Since  $\deg [3D]= \dfrac{s-3}{2}\le 1$,   we get $s\le 5$. If $s=5$, we have the following.  
\vskip 0.2cm\noindent
{\bf (4-C-4)} \quad $g=0, D=E+\dfrac{2}{3}(Q_1+\ldots Q_5), \deg E=-3, 
\quad (3,3,5;15)$. 
\vskip 0.2cm
Next, we assume $s=3, \deg E=-2.$ In this case, $a=3, b\ge 5, \dim R_5= r$ and 
$\dim R_6=r+1$.  
Hence if $r>0, r=1, (a,b,c;h)=(3,5,6;18)$ and 
 $\deg D = \dfrac{1}{5}$. 
\vskip 0.2cm\noindent
{\bf (4-C-5)} \quad $g=0, D=E+\dfrac{2}{3}(Q_1+Q_2+ Q_3)+\dfrac{1}{5}P, 
\deg E=-2, \quad (3,5,6;18)$. 
\vskip 0.2cm
If $r=0$, $\deg [7D] = -2 +t$, where $t$ is the number of $q_i's$ with 
$q_i\ge 7$.  We have the following cases; 
\vskip 0.2cm\noindent
{\bf (4-C-6)} \quad $g=0, D=E+\dfrac{5}{7}(Q_1+Q_2+ Q_3), 
\deg E=-2, \quad (3,7,7;21)$. 
\vskip 0.2cm\noindent
{\bf (4-C-7)} \quad $g=0, D=E+\dfrac{2}{3}Q_1+\dfrac{5}{7}(Q_2+ Q_3), 
\deg E=-2, \quad (3,7,14;28)$. 
\vskip 0.2cm\noindent
{\bf (4-C-8)} \quad $g=0, D=E+\dfrac{2}{3}Q_1+\dfrac{5}{7}Q_2+\dfrac{8}{11}Q_3,
\deg E=-2, \quad (3,7,11;25)$. 
\vskip 0.2cm\noindent
{\bf (4-C-9)} \quad $g=0, D=E+\dfrac{2}{3}(Q_1+Q_2)+\dfrac{11}{15}Q_3, 
\deg E=-2, \quad (3,11,15;33)$. 
\vskip 0.2cm\noindent
{\bf (4-C-10)} \quad $g=0, D=E+\dfrac{2}{3}(Q_1+Q_2)+\dfrac{8}{11}Q_3, 
\deg E=-2, \quad (3,11,18;36)$. 
\vskip 0.2cm\noindent
{\bf (4-C-11)} \quad $g=0, D=E+\dfrac{2}{3}(Q_1+Q_2)+\dfrac{5}{7}Q_3, 
\deg E=-2, \quad (3,14,21;42)$. 
\vskip 0.2cm
Then we treat the case $s=1,\deg E=-1$.  Since $s+r\ge 3$, $r\ge 2$ and 
we have $a=5, b=6$ and the only possible case is;   
\vskip 0.2cm\noindent
{\bf (4-C-12)} \quad $g=0, D=E+\dfrac{2}{3}Q+\dfrac{1}{5}(P_1+P_2), 
\deg E=-1, \quad (5,6,15;30)$. 
\vskip 0.2 cm
This finishes the classification of the case with $a(R)=4$.

\vskip 0.8cm
\centerline{\large\bf The case  $\alpha=5$}
\vskip 0.8cm

We  assume that $R=R(X,D)\cong k[u,v,w]/(f)$ with 
$$\deg(u,v,w;f)=(a,b,c;h); \quad h=a+b+c+5.$$

Since $5D$ is linearly equivalent to $K_X+\FRC(D)$,   we may assume that

$$\leqno{(5.1.1)}\quad \begin{array}{rcl}
D& = & E+\sum_{i=1}^r{}_{q_i \equiv 1 (\mod 5)} \dfrac{q_i-1}{5q_i}P_i+
\sum_{i=1}^s {}_{q_i \equiv 3 (\mod 5)}\dfrac{2q_j-1}{5q_j}P'_j\\
& & +\sum_{k=1}^t {}_{q_k \equiv 2 (\mod 5)}\dfrac{3q_k-1}{5q_k}Q_k+
\sum_{l=1}^u {}_{q_l \equiv 4 (\mod 5)}\dfrac{4q_l-1}{5q_l}Q'_l,
\end{array}$$
where $E$ is an integral divisor on $X$. 

Since $5D\sim $ $K_X+\FRC(D)$, we have 
$$\leqno{(5.1.2)} \quad \begin{array}{cl} &
5E +  \sum_{j=1}^s P'_j+2\sum_{k=1}^t Q_k+3\sum_{l=1}^t Q'_l
 \sim K_X;\\
& \deg E=\dfrac{2g-2-s-2t-3u}{5}\ge \dfrac{2g-2}{5}+\dfrac{r}{6}+
\dfrac{2s}{15}+\dfrac{t}{10}+\dfrac{3u}{20}.
\end{array}$$
\vskip 0.3cm
We divide the cases according to (A) $a=b=1$, 
(B) $a=1,b\ge 2$, (C) $a\ge 2$. 

\vskip 0.5cm
{\bf Case A.  $a=b=1$.}
\vskip 0.5cm
In this case, the type of $R$ is the form $(1,1,c; c+7)$. 
By \ref{bd-of-c}, $c\le 7$ and we have the following cases.
\vskip 0.2cm\noindent
{\bf (5-A-1)} \quad $g=21$, $D=E$ with $\deg E=8$, $5E\sim K_X$ 
;\quad $(1,1,1;8).$ 
\vskip 0.2cm\noindent
{\bf (5-A-2)} \quad $g=12,D=E+\dfrac{1}{2}P$ with 
$\deg E=4, 5E+2P\sim K_X$; \quad $(1,1,2;9).$ 
\vskip 0.2cm\noindent
{\bf (5-A-3)} \quad $g=9,D=E+\dfrac{1}{3}P$ with $\deg E=3, 5E+P\sim K_X$; 
\quad $(1,1,3;10).$ 
\vskip 0.2cm\noindent
{\bf (5-A-4)} \quad $g=6,D=E+\dfrac{1}{6}P$ with $\deg E=2, 5E\sim K_X$; 
\quad $(1,1,6;13).$ 
\vskip 0.2cm\noindent
{\bf (5-A-5)} \quad $g=6,D=E$ with $\deg E=2, 5E\sim K_X$; \quad $(1,1,7;14).$ 

\vskip 0.5cm
{\bf Case B.  $a=1,b\ge 2$.}
\vskip 0.5cm
First, we list the cases $a=1$ and $b=2,3$. 
\vskip 0.2cm\noindent
{\bf (5-B-1)} \quad $g=6$, $D=\dfrac{1}{2}(P_1+\ldots +P_5)$ with 
$2(P_1+\ldots +P_5)\sim K_X$;\quad $(1,2,2;10).$ 
\vskip 0.2cm\noindent
{\bf (5-B-2)} \quad $g=5,D=E+\dfrac{1}{2}P$ with 
$\deg E=4, 5E+2P\sim K_X$; \quad $(1,2,3;11).$ 
\vskip 0.2cm\noindent
{\bf (5-B-3)} \quad $g=4$, $D=\dfrac{1}{2}(P_1+P_2 +P_3)$ with  
$2(P_1+P_2 +P_3)\sim K_X$ ;\quad $(1,2,4;12).$ 
\vskip 0.2cm\noindent
{\bf (5-B-4)} \quad $g=3,D=\dfrac{1}{2}P_1+\dfrac{1}{3}(P_2 +P_3)$ with 
$2P_1+P_2 +P_3\sim K_X$; \quad $(1,2,6;14).$ 
\vskip 0.2cm\noindent
{\bf (5-B-5)} \quad $g=3$, $D=\dfrac{1}{2}P_1+\dfrac{4}{7}P_2$ 
with $2(P_1+P_2)\sim K_X$ ;\quad $(1,2,7;15).$ 
\vskip 0.2cm\noindent
{\bf (5-B-6)} \quad $g=3,D=\dfrac{1}{2}(P_1+P_2)$ with  
$2(P_1+P_2 )\sim K_X$; \quad $(1,2,8;16).$ 
\vskip 0.2cm\noindent
{\bf (5-B-7)} \quad $g=3$, $D=\dfrac{1}{3}(P_1+\ldots +P_4)$ with 
$(P_1+\ldots +P_4)\sim K_X$;\quad $(1,3,3;12).$ 
\vskip 0.2cm\noindent
{\bf (5-B-8)}  \quad $g=3,D=\dfrac{1}{3}P_1+\dfrac{3}{4}P_2$ with 
$P_1+3P_2\sim K_X$; \quad $(1,3,4;13).$ 
\vskip 0.2cm\noindent
{\bf (5-B-9)} \quad $g=2$, $D=\dfrac{1}{3}(P_1+P_2)+\dfrac{1}{6}P_3$ with 
$P_1+P_2\sim K_X$;\quad $(1,3,6;15).$ 
\vskip 0.2cm\noindent
{\bf (5-B-10)} \quad $g=2,D=\dfrac{1}{3}(P_1+P_2)$ with 
$P_1+P_2\sim K_X$; \quad $(1,3,9;18).$ 
\vskip 0.2cm
If $a=1$ and $b\ge 4$, then $\deg D<1$ and $g\le 2$. If $g=2=\dim R_5$, 
then $b\le 5$.
\vskip 0.2cm\noindent
{\bf (5-B-11)} \quad $g=2,D=\dfrac{1}{2}P_1+\dfrac{1}{6}P_2$ with 
$2P_1\sim K_X$; \quad $(1,4,6;16).$ 
\vskip 0.2cm\noindent
{\bf (5-B-12)} \quad $g=2,D=\dfrac{1}{2}P$ with 
$2P\sim K_X$; \quad $(1,4,10;20).$ 
\vskip 0.2cm
In the case $a=1$ and $b\ge 6$, then $D\ge 0, \deg D\le \dfrac{1}{2}$ 
and $g=1$. 
\vskip 0.2cm\noindent
{\bf (5-B-13)} \quad $g=1,D=\dfrac{1}{6}(P_1+P_2+P_3)$; \quad $(1,6,6;18).$ 
\vskip 0.2cm\noindent
{\bf (5-B-14)} \quad $g=1,D=\dfrac{1}{6}(P_1+P_2)$;\quad $(1,6,12;24).$ 
\vskip 0.2cm
If $a=1$ and $b\ge 7$, then $\deg D< \dfrac{1}{3}$. Hence only possibility 
is of the form $D=\dfrac{q-1}{5q}$ with $q\equiv 1$ (mod $5$). 
\vskip 0.2cm\noindent
{\bf (5-B-15)} \quad $g=1,D=\dfrac{2}{11}$; \quad $(1,11,17;34).$ 
\vskip 0.2cm\noindent
{\bf (5-B-16)} \quad $g=1,D=\dfrac{3}{16}P$; \quad $(1,11,16;33).$
\vskip 0.2cm\noindent
{\bf (5-B-17)} \quad $g=1,D=\dfrac{1}{6}P$; \quad $(1,12,18;36).$
\vskip 0.2cm
This finishes the case  $a=1$.
\vskip 0.5cm
{\bf Case C.  $a\ge 2$.}
\vskip 0.5cm
In this case, $\deg D \le \dfrac{12}{2\cdot 2\cdot 3}=1$. 
Since $g=\dim R_5$, $g\le 2$ and $g=2$ only in the following cases.
\vskip 0.2cm\noindent
{\bf (5-C-1)} \quad $g=2,D=E+\dfrac{1}{2}(P_1+P_2)$, with 
$E\ne 0, 2E+P_1+P_2\sim K_X$; \quad $(2,2,3;12).$
\vskip 0.2cm\noindent
{\bf (5-C-2)} \quad $g=2,D=\dfrac{1}{6}P$; \quad $(2,3,5;15).$
\vskip 0.2cm
If $g=1$, $a=2$ then we have the following cases.
\vskip 0.2cm\noindent
{\bf (5-C-3)} \quad $g=1,D=E+\dfrac{1}{2}P_1+\dfrac{1}{2}P_2+
\dfrac{1}{3}P_3+E$, with 
$\deg E=-1, 2E+P_1+P_2\sim E+P_3\sim 0$; \quad $(2,3,7;17).$
\vskip 0.2cm\noindent
{\bf (5-C-4)} \quad $g=1,D=\dfrac{3}{8}P+E$ with 
$2E\sim 0, 5E+P\sim 0; \quad (2,3,8;18).$
\vskip 0.2cm\noindent
{\bf (5-C-5)} \quad $g=1,D=\dfrac{1}{3}P+E$ with 
$2E\sim 0, 5E+P\sim 0; \quad (2,3,10;20).$
\vskip 0.2cm
If $g=1$ and $a\ge 3$, we have $5E\sim 0$ since $\dim R_5=1$ 
and by (6.1.2), $s=t=u=0$. Since $E\ne 0$, $R_n=0$ for $1\le n\le 4$
and we have $a=5,b\ge 6$.  Hence $\deg D\le \dfrac{22}{5\cdot 6\cdot 6}
<\dfrac{1}{6}$.  On the other hand, since $r>0$, we should have 
$\deg E\ge \dfrac{1}{6}$.  Hence there is no case with $g=1,a\ge 3$.
\vskip 0.2cm
Now we treat the case $g=0$.  We divide our discussion into the following cases  
\vskip 0.2cm
\centerline{(i) $R_2\ne 0$, \quad (ii) $R_2=0$ and $R_3\ne 0$, \quad 
(iii) $R_2=R_3=0$. }
\vskip 0.2cm
(i) If $\dim R_2=2$, we have the following cases.
\vskip 0.2cm\noindent
{\bf (5-C-6)} \quad $g=0,D=E+\dfrac{1}{2}
(P_1+\ldots +P_8)+\dfrac{4}{7}P_9$ with 
$\deg E=-4; \quad (2,2,7;16).$
\vskip 0.2cm\noindent
{\bf (5-C-7)} \quad $g=0,D=E+\dfrac{1}{2}
(P_1+\ldots +P_9)$ with 
$\deg E=-4; \quad (2,2,9;18).$
\vskip 0.2cm
If $a=2,b=4$, we have the following cases. 
\vskip 0.2cm\noindent
{\bf (5-C-8)} \quad $g=0,D=E+\dfrac{1}{2}
(P_1+\ldots +P_4)+\dfrac{4}{7}P_5+\dfrac{3}{4}P_6$ with 
$\deg E=-3; \quad (2,4,7;18).$
\vskip 0.2cm\noindent
{\bf (5-C-9)} \quad $g=0,D=E+\dfrac{1}{2}
(P_1+\ldots +P_5)+\dfrac{7}{9}P_6$ with 
$\deg E=-3; \quad (2,4,9;20).$
\vskip 0.2cm\noindent
{\bf (5-C-10)} \quad $g=0,D=E+\dfrac{1}{2}
(P_1+\ldots +P_5)+\dfrac{3}{4}P_6$ with 
$\deg E=-3; \quad (2,4,11;22).$
\vskip 0.2cm
If $a=2$ and $b\ge 6$, from $\deg[2D]=\deg[4D]=0$ and 
$\deg[3D]=-2$ we have $\deg E=-2, t=4, s=u=0$. $b=6$ 
if and only if $r>0$.  
\vskip 0.2cm\noindent
{\bf (5-C-11)} \quad $g=0,D=E+\dfrac{1}{2}
(P_1+\ldots +P_3)+\dfrac{4}{7}P_4+\dfrac{1}{6}P_5$ with 
$\deg E=-2; \quad (2,6,7;20).$
\vskip 0.2cm\noindent
{\bf (5-C-12)} \quad $g=0,D=E+\dfrac{1}{2}
(P_1+\ldots +P_4)+\dfrac{2}{11}P_5$ with 
$\deg E=-2; \quad (2,6,11;24).$
\vskip 0.2cm\noindent
{\bf (5-C-13)} \quad $g=0,D=E+\dfrac{1}{2}
(P_1+\ldots +P_4)+\dfrac{1}{6}P_5$ with 
$\deg E=-2; \quad (2,6,13;26).$
\vskip 0.2cm
If $a=2$ and $b\ge 7$, then we have $\deg E=-2, t=4, r=s=u=0$. 
\vskip 0.2cm\noindent
{\bf (5-C-14)} \quad $g=0,D=E+\dfrac{1}{2}P_1+\dfrac{4}{7}
(P_2+P_3 +P_4)$ with 
$\deg E=-2; \quad (2,7,7;21).$
\vskip 0.2cm\noindent
{\bf (5-C-15)} \quad $g=0,D=E+\dfrac{1}{2}
(P_1+P_2)+\dfrac{4}{7}(P_3 +P_4)$ with
$\deg E=-2; \quad (2,7,14;28).$
\vskip 0.2cm\noindent
{\bf (5-C-16)} \quad $g=0,D=E+\dfrac{1}{2}
(P_1+P_2)+\dfrac{4}{7}(P_3 +P_4)$ with 
$\deg E=-2; \quad (2,7,14;28).$
\vskip 0.2cm\noindent
{\bf (5-C-17)} \quad $g=0,D=E+\dfrac{1}{2}
(P_1+P_2+P_3)+\dfrac{10}{17}P_4$ with 
$\deg E=-2; \quad (2,12,17;36).$
\vskip 0.2cm\noindent
{\bf (5-C-18)} \quad $g=0,D=E+\dfrac{1}{2}
(P_1+P_2+P_3)+\dfrac{7}{12}P_4$ with 
$\deg E=-2; \quad (2,12,19;38).$
\vskip 0.2cm\noindent
{\bf (5-C-19)} \quad $g=0,D=E+\dfrac{1}{2}
(P_1+P_2+P_3)+\dfrac{4}{7}P_4$ with 
$\deg E=-2; \quad (2,14,21;42).$
\vskip 0.2cm
This finishes the case $g=0$ and $a=2$.
\vskip 0.2cm
{\bf (ii) The case  $g=0$ and $a=3$.}
\vskip 0.2cm
If $a=3$ and $b=3$ or $4$, we have the following cases.
\vskip 0.2cm\noindent
{\bf (5-C-20)} \quad $g=0,D=E+\dfrac{1}{3}
(P_1+\ldots+P_5)+\dfrac{3}{4}P_6$ with 
$\deg E=-2; \quad (3,3,4;15).$
\vskip 0.2cm\noindent
{\bf (5-C-21)} \quad $g=0,D=E+\dfrac{3}{4}
(P_1+\ldots+P_4)+\dfrac{1}{3}P_5$ with 
$\deg E=-3; \quad (3,4,4;16).$
\vskip 0.2cm\noindent
{\bf (5-C-22)} \quad $g=0,D=E+\dfrac{1}{3}
P_1+\dfrac{3}{4}(P_2+P_3)+\dfrac{3}{8}P_4$ with 
$\deg E=-2; \quad (3,4,8;20).$
\vskip 0.2cm\noindent
{\bf (5-C-23)}  \quad $g=0,D=E+\dfrac{1}{3}
(P_1+P_2)+\dfrac{3}{4}P_3+\dfrac{7}{9}P_4$ with 
$\deg E=-2; \quad (3,4,9;21).$
\vskip 0.2cm\noindent
{\bf (5-C-24)}  \quad $g=0,D=E+\dfrac{1}{3}
(P_1+P_2)+\dfrac{3}{4}(P_3+P_4)$ with 
$\deg E=-2; \quad (3,4,12;24).$
\vskip 0.2cm
If $a=3$ and $b\ge 6$, then since $\deg[D]=\deg[4D]=-1$ 
by Lemma 0.3 and $\deg[3D]=0$, we have $\deg E=-1, s=3, t=u=0$.
 Also, in this case $r>0$ if and only if $b=6$.  In the latter 
case, since $\deg D\le \dfrac{21}{3\cdot 6\cdot 7}=\dfrac{1}{6}$,
$D=E+\dfrac{1}{3}(P_1+P_2+P_3)+\dfrac{1}{6}P_4$ with 
$\deg E=-1$.
\vskip 0.2cm\noindent
{\bf (5-C-25)}  \quad $g=0,
D=E+\dfrac{1}{3}(P_1+P_2+P_3)+\dfrac{1}{6}P_4$ with 
$\deg E=-1; \quad (3,6,7;21).$
\vskip 0.2cm
If $r=0$, we have the following cases.
\vskip 0.2cm\noindent
{\bf (5-C-26)}  \quad $g=0,
D=E+\dfrac{3}{8}(P_1+P_2+P_3)$ with 
$\deg E=-1; \quad (3,8,8;24).$
\vskip 0.2cm\noindent
{\bf (5-C-27)}  \quad $g=0,
D=E+\dfrac{1}{3}P_1+\dfrac{3}{8}P_2+\dfrac{5}{13}P_3$ with 
$\deg E=-1; \quad (3,8,13;29).$
\vskip 0.2cm\noindent
{\bf (5-C-28)}  \quad $g=0,
D=E+\dfrac{1}{3}P_1+\dfrac{3}{8}(P_2+P_3)$ with 
$\deg E=-1; \quad (3,8,16;32).$
\vskip 0.2cm\noindent
{\bf (5-C-29)}  \quad $g=0,
D=E+\dfrac{1}{3}(P_1+P_2)+\dfrac{7}{18}P_3$ with 
$\deg E=-1; \quad (3,13,18;39).$
\vskip 0.2cm\noindent
{\bf (5-C-30)}  \quad $g=0,
D=E+\dfrac{1}{3}(P_1+P_2)+\dfrac{5}{13}P_3$ with 
$\deg E=-1; \quad (3,13,21;42).$
\vskip 0.2cm\noindent
{\bf (5-C-31)}  \quad $g=0,
D=E+\dfrac{1}{3}(P_1+P_2)+\dfrac{3}{8}P_3$ with 
$\deg E=-1; \quad (3,16,24;48).$
\vskip 0.2cm
This finishes the case $g=0$ and $a=3$.   
\vskip 0.2cm
{\bf (iii) $g=0$, $R_2=R_3=0$.} 
\vskip 0.2cm
First we treat the case $R_4\ne 0$.  It is easy to see that 
hypersurfaces of type $(4,4,c;c+13)$ can not be normal. 
So, if $a=4$, then $n\ge 6$ and we have $\deg E=-2, s+u=2, t+u=3$.
If $a=4$ and $b=6$, we have the following types.
\vskip 0.2cm\noindent
{\bf (5-C-32)}  \quad $g=0,
D=E+\dfrac{1}{2}(P_1+P_2)+\dfrac{1}{3}P_3+\dfrac{7}{9}P_4$ with 
$\deg E=-2; \quad (4,6,9;24).$
\vskip 0.2cm\noindent
{\bf (5-C-33)}  \quad $g=0,
D=E+\dfrac{1}{2}(P_1+P_2)+\dfrac{1}{3}P_3+\dfrac{3}{4}P_4$ with 
$\deg E=-2; \quad (4,6,15;30).$
\vskip 0.2cm
If $a=4$ and $b>6$, then $\deg [6D]<0$. Since $5D \sim K_X+\FRC(D)$, this is 
only possible if the support of $\FRC(D)$ consists of $3$ points. This implies 
$r=s=0, t=1,u=2$.  
\vskip 0.2cm\noindent 
{\bf (5-C-34)}  \quad $g=0,
D=E+\dfrac{1}{2}P_1+\dfrac{7}{9}P_2+\dfrac{11}{14}P_3$ with 
$\deg E=-2; \quad (4,9,14;32).$
\vskip 0.2cm\noindent
{\bf (5-C-35)}  \quad $g=0,
D=E+\dfrac{1}{2}P_1+\dfrac{7}{9}(P_2+P_3)$ with $\deg E=-2; \quad (4,9,18;36).$
\vskip 0.2cm\noindent
{\bf (5-C-36)}  \quad $g=0,
D=E+\dfrac{1}{2}P_1+\dfrac{3}{4}P_2+\dfrac{7}{9}P_3$ with $\deg E=-2;
 \quad (4,18,27;54).$
\vskip 0.2cm\noindent
{\bf (5-C-37)}  \quad $g=0,
D=E+\dfrac{1}{2}P_1+\dfrac{3}{4}P_2+\dfrac{11}{14}P_3$ with 
$\deg E=-2; \quad (4,14,23;46).$
\vskip 0.2cm\noindent
{\bf (5-C-38)}  \quad $g=0,
D=E+\dfrac{1}{2}P_1+\dfrac{3}{4}P_2+\dfrac{15}{19}P_3$ with 
$\deg E=-2; \quad (4,14,19;42).$
\vskip 0.2cm
This finishes the case $a=4$. If $a>4$, then $a\ge 6$ and 
since $\deg[nD]=-1$ by Lemma 1.4, we have $\deg E=-1$. 
Since $6D =(K_X+\FRC(D)) +D$, $\deg [6D] = -3+r+s+t+u\ge 0$. 
Hence $a=6$ and it is easy to see type $(6,6,c; c+17)$ does 
not occur.  Hence we have either $r=2,s=t=0,u=1$ or 
$r=s=t=1, u=0$. In the former case, $\deg D\ge 2\dfrac{1}{6}+\dfrac{3}{4}
-1=\dfrac{1}{12}$. On the other hand, since $b\ge 7$ and $c\ge 8$, 
$\deg D\le \dfrac{26}{6\cdot 7\cdot 8}<\dfrac{1}{12}$. So, this case 
does not occur and we have $r=s=t=1,u=0$. We have the following cases.
\vskip 0.2cm\noindent
{\bf (5-C-39)}  \quad $g=0,
D=E+\dfrac{1}{3}P_1+\dfrac{1}{6}P_2+\dfrac{4}{7}P_3$  with 
$\deg E=-1; \quad (6,7,9;27).$
\vskip 0.2cm\noindent
{\bf (5-C-40)}  \quad $g=0,
D=E+\dfrac{1}{2}P_1+\dfrac{3}{8}P_2+\dfrac{2}{11}P_3$
with $\deg E=-1; \quad (6,8,11;30).$
\vskip 0.2cm\noindent
{\bf (5-C-41)}  \quad $g=0,
D=E+\dfrac{1}{2}P_1+\dfrac{1}{6}P_2+\dfrac{5}{13}P_3$
with $\deg E=-1; \quad (6,8,13;32).$
\vskip 0.2cm\noindent
{\bf (5-C-42)}  \quad $g=0,
D=E+\dfrac{1}{2}P_1+\dfrac{1}{3}P_2+\dfrac{2}{11}P_3$
with $\deg E=-1; \quad (6,22,33;66).$
\vskip 0.2cm\noindent
{\bf (5-C-43)}  \quad $g=0,
D=E+\dfrac{1}{2}P_1+\dfrac{1}{3}P_2+\dfrac{3}{16}P_3$
with $\deg E=-1; \quad (6,16,27;54).$
\vskip 0.2cm\noindent
{\bf (5-C-44)}  \quad $g=0,
D=E+\dfrac{1}{2}P_1+\dfrac{1}{3}P_2+\dfrac{4}{21}P_3$
with $\deg E=-1; \quad (6,16,21;48).$
\vskip 0.2cm
This finishes the ccase $a(R)=5$.

\vskip 0.8cm
\centerline{\large\bf The case  $\alpha=6$}
\vskip 0.8cm

Since $6D$ is linearly equivalent to $K_X+\FRC(D)$,     we may assume that

$$\leqno{(6.1.1)}\quad 
D=E+\sum_{i=1}^r{}_{q_i \equiv 1 (\mod 6)} \dfrac{q_i-1}{6q_i}P_i+
\sum_{i=1}^s {}_{q_i \equiv 5 (\mod 6)}\dfrac{5q_i-1}{6q_i}Q_i,$$
where $E$ is an integral divisor on $X$. 

Since $6D\sim K_X+\FRC(D)$, we have 
$$\leqno{(6.1.2)} \quad 
6E + 4\sum_{i=1}^s Q_i \sim K_X; \quad \deg E=\dfrac{g-1-2s}{3}.$$

Since $\dfrac{q_i-1}{6q_i}\ge \dfrac{1}{7}$ if $q_i \equiv 1 (\mod 6), 
q_i>1$  and $\dfrac{5q_i-1}{6q_i}\ge \dfrac{4}{5}$ if $q_i \equiv 5
 (\mod 6)$,we have 
$$\leqno{(6.1.3)}\quad 
\deg D\ge \deg E+\dfrac{r}{7}+\dfrac{4s}{5}=\dfrac{g-1}{3}+
\dfrac{r}{7}+\dfrac{2s}{15}.$$

\vskip 0.3cm
We divide the cases according to (A) $a=b=1$, 
(B) $a=1,b\ge 2$, (C) $a\ge 2$.

\vskip 0.5cm
{\bf Case A.  $a=b=1$.}
\vskip 0.5cm
In this case, we have the following cases. 
We can calculate the genus $g$ by 
$g=\dim R_6$. 
\vskip 0.2cm\noindent
{\bf (6-A-1)} \quad $g=28$, $D=E$ with $\deg E=9$, $6E\sim K_X$ 
;\quad $(1,1,1;9).$ $X$ is not hyperelliptic.
\vskip 0.2cm\noindent
{\bf (6-A-2)} \quad $g=16$, $D=E$ with $\deg E=5$, $6E\sim K_X$;
\quad $(1,1,2;10).$ $X$ is not hyperelliptic.
\vskip 0.2cm\noindent
{\bf (6-A-3)} \quad $g=10$, $D=E$ with $\deg E=3$, $6E\sim K_X$;
\quad $(1,1,4;12).$ 
\vskip 0.2cm\noindent
{\bf (6-A-4)} \quad $g=7$, $D=E+\dfrac{1}{7}$ with $\deg E=2$, 
$6E\sim K_X$;\quad $(1,1,7;15).$ 
\vskip 0.2cm\noindent
{\bf (6-A-5)} \quad $g=7$, $D=E$ with $\deg E=2$, $6E\sim K_X$ 
;\quad $(1,1,8;16).$ 
\vskip 0.5cm
{\bf Case B.  $a=1, b\ge 2$.}
\vskip 0.5cm
First we list the cases with $a=1,b=2,3,4$.
\vskip 0.2cm\noindent
{\bf (6-B-1)} \quad $g=7$, $D=E$ with $\deg E=2$, 
$6E\sim K_X$;\quad $(1,2,3;12).$ 
\vskip 0.2cm\noindent
{\bf (6-B-2)} \quad $g=4$, $D=E+\dfrac{1}{7}$ with $\deg E=1$, 
$6E\sim K_X$;\quad $(1,2,7;16).$ 
\vskip 0.2cm\noindent
{\bf (6-B-3)} \quad $g=4$, $X$ is hyperelliptic, $D=E$ with $\deg E=1$, 
$6E\sim K_X$;\quad $(1,2,9;18).$ 
\vskip 0.2cm\noindent
{\bf (6-B-4)} \quad $g=4$, $D=E$ with $\deg E=1$, 
$6E\sim K_X$;\quad $(1,3,5;15).$ 
\vskip 0.2cm\noindent
{\bf (6-B-5)} \quad $g=3$, $D=\dfrac{4}{5}P$ with  
$4P\sim K_X$;\quad $(1,4,5;16).$ 
\vskip 0.2cm
If $b\ge 5$, then $\deg D\le \dfrac{3}{5}$.  Hence $E=0$ 
and $s=0$.  This implies $\deg [nD] =0$ for $n\le 6$. 
Thus we have $g=1$ and $b\ge 7$, $\deg D\le \dfrac{3}{7}$.   
\vskip 0.2cm
\noindent
{\bf (6-B-6)} \quad $g=1$, $D=\dfrac{1}{7}(P_1+P_2+P_3); \quad (1,7,7;21).$ 
\vskip 0.2cm
\noindent
{\bf (6-B-7)} \quad $g=1$, $D=\dfrac{1}{7}(P_1+P_2+P_3); \quad (1,7,7;21).$ 
\vskip 0.2cm
\noindent
{\bf (6-B-6)} \quad $g=1$, $D=\dfrac{1}{7}(P_1+P_2); \quad (1,7,14;28).$ 
\vskip 0.2cm
\noindent
{\bf (6-B-7)} \quad $g=1$, $D=\dfrac{1}{7}P_1+\dfrac{2}{13}P_2; \quad (1,7,13;27).$
\vskip 0.2cm 
\noindent
{\bf (6-B-8)} \quad $g=1$, $D=\dfrac{3}{19}P; \quad (1,13,19;39).$ 
\vskip 0.2cm
\noindent
{\bf (6-B-9)} \quad $g=1$, $D=\dfrac{2}{13}P; \quad (1,13,20;40).$ 
\vskip 0.2cm
\noindent
{\bf (6-B-10)} \quad $g=1$, $D=\dfrac{1}{7}P; \quad (1,14,21;42).$ 
\vskip 0.2cm
This finishes the case $a=1,b\ge 2$.
\vskip 0.5cm
{\bf Case C.  $a\ge 2$.}
\vskip 0.5cm
We can easily see that $(a,b)=(2,2),(2,3),(2,4),(3,3)$ does not occur. 
By a similar computation, we can assert $\deg D<\dfrac{1}{3}$ and 
we know that $g\le 1$. 
\vskip 0.2cm
If $g=1$, then $6E+4\sum_{i-1}^s Q_j\sim 0$. If $\deg E=-2$, then 
$s=3$ and $\deg D\ge \dfrac{2}{5}$, which contradicts our previous 
computation.  Hence, if $g=1$, then $\deg E=0$ and $s=0$.
 This implies that $a=2,3$ or $6$ and $\deg[nD]=0$ for $1\le n\le 6$.    
We have the following cases.
\vskip 0.2cm
\noindent
{\bf (6-C-1)} \quad $g=1$, $D=E+\dfrac{1}{7}P$ with $E\ne 0$, 
$2E\sim 0; \quad (2,7,15;30).$ 
\vskip 0.2cm
\noindent
{\bf (6-C-2)} \quad $g=1$, $D=E+\dfrac{2}{13}P$ with $E\ne 0$, 
$2E\sim 0; \quad (2,7,13;28).$ 
\vskip 0.2cm
\noindent
{\bf (6-C-3)} \quad $g=1$, $D=E+\dfrac{1}{7}P$ with $E\ne 0$, 
$3E\sim 0; \quad (3,7,8;24).$ 
\vskip 0.2cm
This finishes the case $g=1$. If $g=0$, we have 
$$6E + 4\sum_{i-1}^s Q_j\sim K_X, \quad \deg E=\dfrac{-1-2s}{3}.$$

If $deg E=-1, s=1$, then $R_n=0$ for $n\le 6$ and this implies 
$\deg D\le \dfrac{28}{7\cdot 7\cdot 8}=\dfrac{1}{14}$.  
On the other hand, since $r+s\ge 3$, $\deg D\ge -1+\dfrac{4}{5}
+\dfrac{2}{7}=\dfrac{6}{35}$, a contradiction !
\vskip 0.2cm
Hence we have $s=4$ and $\deg E=-3$.  Then $\deg [4D]=0$ 
and $\deg [5D] =1$.  Hence we are restricted to the following type.
\vskip 0.2cm
\noindent
{\bf (6-C-4)} \quad $g=0$, $D=E+\dfrac{4}{5}(P_1+\ldots +P_4)$ with 
$\deg E=-3;  \quad (4,5,5;20).$ 
\vskip 0.2cm
This finishes the case $a(R)=6$.

\vskip 0.8cm
Now, we list the table of numbers of cases with given $a(R)\le 6$. 
In the table, br denotes the number of brances of $D$, which is equal 
to the number of branches of the graph of the minimal resolution of 
$\Spec(R)$. In other word, br $=\deg \lceil\FRC(D)\rceil$. 
\vskip 0.8cm
\begin{center}
 {\large 
\begin{tabular}{|c|c|c|c|c|c|c|}\hline
$a(R)$        & 1   & 2 & 3 & 4 & 5 & 6  \\ \hline
\phantom{xxxxxxxxx} &\phantom{xxxxx} &\phantom{xxxxx} &\phantom{xxxxx} &
\phantom{xxxxx} &\phantom{xxxxx} &\phantom{xxxxx} \\ \hline
$g=0$, br$=3$  & 14  & 6  & 7 & 7  &  11 & 0   \\ \hline
$g=0$, br$\ge4$& 8   & 1  & 10& 2  &  22  & 1   \\ \hline
$g=1$       & 6   & 8  & 8 & 7  &   8 &  8   \\ \hline
$g=2$       & 2   & 2  & 3 & 3  & 6  & 0   \\ \hline  
$g=3$       & 1   & 2  & 2 & 3  & 5  & 1    \\ \hline
$g\ge 4$    & 0   & 2  & 4 & 6  & 8  & 9   \\ \hline
Total       & 31  &21  & 34& 28 &  58 &  19   \\ \hline
\end{tabular}}
\end{center}


\vskip 1cm
\begin{thebibliography}{Dem}
\bibitem[Dem]{Dem}
Demazure, M.:  
              Anneaux gradu{\'e}s normaux 
             ; in Seminarire Demazure -Giraud -Teissier, 1979. 
                Ecole Polytechnique
              In:L{\^e} D{\~u}ng Tr{\'a}ng(ed.) 
          Introduction  a la th{\'e}orie des singularit{\'e}s II; 
            M{\'e}thodes alg{\'e}briques et g{\'e}om{\'e}triques  
        ( Travaux En Cours vol.{\bf 37}, pp.35-68) Paris:Hermann 1988

\bibitem[GW]{GW}  
Goto, S.,   Watanabe, K.-i.:
        On graded rings I.
        J. Math. Soc. Japan. {\bf 30} (1978) 179-213.
 
 \bibitem[Dol]{Dol} Dolga\v cev, I.V., Automorphic forms, and quasihomogeneous 
 singularities, Funcional Anal. Appl. {\bf 9} (1975), 149-151.
       

\bibitem[OW]{OW}
Orlik, P. and Wagreich, P., Isolated singularities of algebraic surfaces with $\C^*$-action, 
Annals. Math., {\bf 93} (1971), 205-228. 

\bibitem[P1]{P1}
Pinkham, H.: Singularit\'es exceptionelles, la dualit\'e etrange d'Arnold 
et les surfaces K3, C.R.Acad.Sc.Paris, {\bf 284} (A) (1977), 615-618.        
        
\bibitem[P2]{P2}
Pinkham, H.:  
        Normal surface singularities with ${\bold C}^{*}$-action   
        Math. Ann. {\bf 227}, (1977) 183-193

\bibitem[S0]{S0}  Saito, K.:
Quasihomogene isolierte Singularit\" aten von Hyperfl\" achen,
Inventiones Math. {\bf 14} (1971), 123-142.          
\bibitem[S1]{S1}  Saito, K.:  
        Regular system of weights and associated singularities,   
        In:Suwa,T., Wagreich,Ph.(ed.) Complex analytic singularities.
           (Adv. Studies in pure Math.{\bf 8},479-526) Tokyo Amsterdam:
            Kinokuniya-North-Holland 1986  \par\parindent=0pt%

\bibitem[S2]{S2}  Saito, K.:  
   On the existence of exponents prime to the Coxeter number, 
      J. of Algebra, {\bf 114} (1986), 333-356.

\bibitem[To]{To}
M.~Tomari, 
Multiplicity of filtered rings and simple K3 singularities of multiplicity two,
Publ. Res. Inst. Math. Sci. \textbf{38} (2002), 693--724.

\bibitem[Wag]{Wag} P.~Wagreich; Algebras of automorphic forms with few generators. 
Trans. Amer. Math. Soc. {\bf 262} (1980), no. 2, 367?389. 


\bibitem[W]{W}
Watanabe, K.-i.:  
        Some remarks concerning Demazure's construction 
                 of normal graded rings,
        Nagoya Math. J. {\bf 83},(1981) 203-211        
        
\end{thebibliography}
\end{document}